\documentclass[11pt, a4paper]{article}%
\usepackage{eurosym}
\usepackage{amsmath}
\usepackage{amsbsy}
\usepackage{amsfonts}
\usepackage{bm}
\usepackage{amsfonts, graphicx, verbatim, amsmath,amssymb}
\usepackage[hidelinks]{hyperref}
\usepackage{color}
\usepackage{amssymb}
\usepackage{graphicx}
\usepackage{enumitem}
\usepackage{stmaryrd}
\usepackage[numbers]{natbib}
\providecommand{\U}[1]{\protect\rule{.1in}{.1in}}
\providecommand{\U}[1]{\protect\rule{.1in}{.1in}}
\newtheorem{theorem}{Theorem}[section]

\newtheorem{corollary}[theorem]{Corollary}

\newtheorem{example}[theorem]{Example}

\newtheorem{lemma}[theorem]{Lemma}

\newenvironment{proof}[1][Proof]{\textbf{#1.} }{\ \rule{0.5em}{0.5em}\newline}

\newcommand{\Z}{{\mathbb{Z}}}

\newcommand{\Var}{{\rm Var}}

\newcommand{\N}{{\mathbb N}}

\newcommand{\E}{{\mathbb{E}}}
\renewcommand{\P}{{\mathbb{P}}}

\newcommand{\bb}{\begin{eqnarray*}}
	\newcommand{\ee}{\end{eqnarray*}}
\newcommand{\bbb}{\begin{eqnarray}}
	\newcommand{\eee}{\end{eqnarray}}

\def\build#1_#2^#3{\mathrel{\mathop{\kern 0pt#1}\limits_{#2}^{#3}}}
\newcommand{\converge}[3]{\build\hbox to
	15mm{\rightarrowfill}_{#1\rightarrow #2}^{\hbox{\scriptsize #3}}}
\newcommand{\convergedef}[1]{\overset{{\hbox{\scriptsize #1}}}{\hbox to 15mm{\rightarrowfill}}}

\usepackage[left=2cm,right=2cm, top=2cm, bottom=2cm]{geometry}

\begin{document}
	
	\title{On the Quenched Functional Central Limit Theorem for Stationary Random Fields under Projective Criteria}
	\author{{Lucas Reding} and Na Zhang}
	\date{\vspace{-3em}}
	\maketitle
	
	Univ. Polytechnique Hauts-de-France, INSA Hauts-de-France, CERAMATHS, F - 59313 Valenciennes, France
	
	Email: lucas.reding@uphf.fr
	
	Department of Mathematics, Towson University, Towson, MD 21252-0001,USA.
	
	Email: nzhang@towson.edu
	
	\begin{abstract}
		In this work, we study and establish some quenched functional Central Limit Theorems (CLTs) for stationary random fields under a projective criteria. These results are functional generalizations of the theorems obtained by \cite{MR4166203} and of the quenched functional CLTs for ortho-martingales established by \cite{MR4125956} to random fields satisfying a Hannan type projective condition. In the work of \cite{MR4166203}, the authors have already proven a quenched functional CLT, however the assumptions were not optimal as they required the existence of a $2+\delta$-moment. In this article, we establish the results under weaker assumptions, namely we only require an Orlicz space condition to hold. The methods used to obtain these generalizations are somewhat similar to the ones used by \cite{MR4166203} but we improve on them in order to obtain results within the functional framework. Moreover, a Rosenthal type inequality for said Orlicz space is also derived and used to obtain a sufficient condition analogous to that of Theorem 4.4 in the work of \cite{MR4166203}. Finally, we apply our new results to derive some quenched functional CLTs under weak assumptions for a variety of stochastic processes.
		
	\end{abstract}
	
	
	\section{Introduction}
	Developments within the Markovian theory led to the question of the conditions under which a central limit theorem could be derived for Markov chains; in particular what restrictions were sufficient on the initial distribution and the transition operator to have this kind of result. Seminal results were obtained by \cite{MR0501277} (see also \citealp{MR1368394, MR1826405}) for Markov chains endowed with the stationary measure as their initial distribution as well as \cite{MR0834478} (see also \citealp{MR1826405}) for additive functionals of reversible Markov chains. Additionally, \cite{MR1826405} also obtained a CLT for Markov chains starting from a fixed point (in other words, endowed with $\delta _x$, the Dirac measure at the state $x$, as their initial distribution). Such theorems are called quenched CLTs. Another way of expressing these results is to consider a fixed past and to study the convergence in distribution with respect to that past. The difficulties during the proof arise from the fact that this fixed past causes the process to not be stationary anymore. An extensive literature exists on the subject, one can cite the following works by \cite{MR3483741, MR2886384, MR3178473, MR3083921, MR3224292, MR3409830, MR3157895}. Note that some counterexamples to quenched central limit theorems under specific conditions were found by \cite{MR2425365} and \cite{MR2731055}. Functional versions of these quenched central limit theorems, also called quenched weak invariance principles, have also been the subject of numerous research articles such as the ones by \cite{MR3483741, MR3178473, MR3083921, MR3409830}.
	\\
	\\
	Random fields naturally appear as a generalization of sequences of random variables, however extending the one-dimensional results to greater dimensions is much harder than one would think. The first problem we are faced with is to correctly define the notion of past trajectory. The approach we have implemented in this paper is to use the notion of commuting filtrations. In particular, this property is satisfied by filtrations generated by fields of independent random variables or even by fields with independent columns (or, equivalently, independent rows). As a lot of processes can be expressed as a functional of i.i.d. random variables, these types of filtrations are quite common and merit interest. A lot of work has been done under commuting filtrations (see \citealp{MR3427925,MR3504508}).
	\\
	\\
	As usual, we will require some kind of dependency condition on the studied field. Namely, in this paper, we will use Hannan's projective condition as defined by \cite{MR0331683}. The problem we are interested in has been studied by \cite{MR3083921} for time series but it has yet to be investigated for higher dimensions, which is the purpose of this article. Though the problem we focus on hasn't been studied yet, one can note that fields satisfying Hannan's condition have been quite extensively studied and numerous CLTs and functional CLTs, both in the annealed and quenched sense, have been obtained. One could refer to the following works: \cite{MR3264437, MR3483738, MR4166203}.
	\\
	\\ 
	The proofs for the main theorems in this paper are based upon the use of a martingale-coboundary decomposition that can be found in \cite{MR3264437} (some more recent and general results can be found in \citealp{MR3522451, MR3735411, MR3869881}, see also \citealp{MR2749126}) as well as the central limit theorem and the weak invariance principle established by \cite{MR4125956} for ortho-martingales. Once the main theorems are established, we derive corollaries in the spirit of the results obtained by \cite{MR4166203}. As shown by the previous results in the literature, it will be required to address two situations separately: first when the summations are done over cubic regions of $\mathbb{Z}^d$ and, after that, when the regions are only required to be rectangular.
	\\
	\\
	In the previous work of \cite{MR4166203}, the Rosenthal inequality for Lebesgue spaces (see \citealp{MR0624435}, Theorem 2.11, p.23) was used to derive a sufficient condition for the quenched CLT and its functional form to hold. In order to obtain an analogous result within our framework, we will make use of a Rosenthal type inequality for this Orlicz space. Given that no such result seems to exist in the literature, we will follow the outline of the proof given by \cite{MR0365692} and adapt it to our framework in order to establish the required inequality.
	\\
	\\
	This paper will be structured as follows: in Section \ref{sec:framework}, we introduce the notations used throughout our article and we present the main results obtained in this work. In particular, we will split the results into two categories: the first one will aggregate theorems dealing with summations over cubic regions only while the other category will deal with results concerning more general rectangular regions. The proofs of these theorems will appear in Section \ref{sec:proofs} and we will improve on the two applications studied by \cite{MR4166203} as well as provide some additional examples in Section \ref{sec:examples}. These examples include linear and Volterra random fields as well as Hölder continuous functions of linear fields, which are a common occurrence in the field of financial mathematics and economics, and also weakly dependent random fields in the sense of \cite{Wu2005} which hold a significant role in mathematical physics and, in particular, within the study of particle systems. Finally, in Section \ref{sec:appendix}, we give the proof of the Rosenthal type inequality for the Orlicz space mentioned throughout this paper. 
	
	\section{Framework and results}
	\label{sec:framework}
	In all that follows, we consider a probability space $(\Omega, \mathcal{F}, \P)$ and all the random variables considered thereafter will be real-valued and defined on that probability space. We start by introducing multiple items of notation that will be used throughout this article: $d$ will be an integer greater than $1$, $[x]$ will denote the integer part of a real number $x$, bold characters will designate multi-indexes and in particular we shall write $\bm{0}:= (0, \ldots, 0) \in \Z^d$ as well as $\bm{1}:= (1, \ldots, 1) \in \Z^d$. For any $\bm{n} \in \Z^d$, we denote $\bm{n} := (n_1, \ldots,n_d)$ and $|\bm{n}| := \prod_{i=1}^{d}n_i$. The set of all positive integers will be denoted by $\mathbb{N}^*$ and the set of integers $\{1,\ldots, d\}$  will be denoted by $\llbracket1,d \rrbracket$. In order to define the concept of past trajectory, it is necessary to define an order on $\Z^d$: if $\bm{u},\bm{v} \in \Z
	^d$ are multi-indexes such that for all $ k \in \llbracket1,d\rrbracket, u_k \le v_k$, then we will write $\bm{u} \le \bm{v}$.
	\\
	\\
	Convergence of fields indexed by $\Z^d$ will be interpreted in the following sense. If $\bm{n} =(n_1, \ldots, n_d)$ is a multi-index, then the notation $\bm{n} \to \infty$ is to be interpreted as the convergence of $\min \{n_1, \ldots, n_d\}$ to $\infty$. Convergence in distribution (resp. almost surely) will be denoted by $\convergedef{\textrm{$\mathcal{D}$}}$ (resp. $\convergedef{\normalfont{a.s.}}$).
	\\
	\\
	Before introducing the field we are interested in, we define some transformations on $\Omega$. We let $T_i: \Omega \to \Omega$, $i\in\{1, \dots, d\}$ be invertible measure-preserving commuting transforms on the probability space $(\Omega, \mathcal{F}, \P)$ and we make use of the operator notation (i.e. if $U$ and $V$ are two transformations on $\Omega$, we denote $UV:= U\circ V$).
	\\
	\\
	We consider a sigma-field $\mathcal{F}_{\bm{0}} \subset \mathcal{F}$ such that $\mathcal{F}_{\bm{0}} \subset T^{-\bm{i}}\mathcal{F}_{\bm{0}}$ for all $\bm{i} \in \Z^d$, and a random variable $X_{\bm{0}} \in L^0_2$ where $L^0_2 = L^0_2(\Omega, \mathcal{F}_{\bm{0}}, \P)$ is the set of all $\mathcal{F}_{\bm{0}}$-measurable, square integrable, and centered random variables.
	
	For every $\bm{n} = (n_1, \ldots, n_d) \in \mathbb{Z}^d$, set 
	\begin{equation}\label{Xdef}
		X_{\bm{n}} = X_{\bm{0}} \circ T^{\bm{n}},
	\end{equation}
	and 
	\begin{equation}\label{Fdef}
		\mathcal{F}_{\bm{n}} = T^{-\bm{n}}\mathcal{F}_{\bm{0}},
	\end{equation}
	where $T^{\bm{n}} = T_1^{n_1}\cdots T_d^{n_d}$. As a result, $X_{\bm{n}}$ is $\mathcal{F}_{\bm{n}}$-measurable.
	\\
	\\
	Suppose that the family $(\mathcal{F}_{\bm{k}})_{\bm{k}\in \Z^d}$ is a commuting filtration, that is, for every integrable random variable $X$, we have
	\[
	\E_{\bm{i}}\bigl[\E_{\bm{j}}[X]\bigr] = \E_{\bm{i}\wedge \bm{j}}[X],
	\]
	where $\E_{\bm{i}}[X] = \E[X| \mathcal{F}_{\bm{i}}]$ and $\bm{i}\wedge \bm{j}$ is the coordinate-wise minimum between $\bm{i}$ and $\bm{j}$.
	\\
	\\
	We recall the notion of ortho-martingale which was introduced by \cite{MR254912} (see also \citealp{MR1914748}). We say that a random field $(D_{\bm{i}})_{\bm{i} \in \Z^d}$ is an ortho-martingale difference field if each $D_{\bm{n}}$ is in $L^1(\mathcal{F}_{\bm{n}})$ and satisfies the equation $\E_{\bm{a}}[D_{\bm{n}}]=0$ as long as there exists
	$k \in \llbracket 1,d \rrbracket$ such that $a_k<n_k$. Then, if $M_{\bm{n}}:=\sum\nolimits_{\bm{0} \le \bm{u} \le \bm{n}}D_{\bm{u}}$, the random field $(M_{\bm{n}})_{\bm{n} \in \N^d}$ will be called an ortho-martingale.
	\\
	\\
	Suppose also that the random variable $X_{\bm{0}}$ is regular with respect to the filtration $\mathcal{F}$, that is $\E[X_{\bm{0}}| \mathcal{F}_{-\infty\bm{e}_i}] = 0$ for every $i \in \{1,\cdots,d\}$, where $\bm{e}_i$ is the multi-index whose $i$-th coordinate is equal to $1$ and the others are equal to $0$ with the convention that $\infty \times 0 = 0$.
	\\
	\\
	We consider the projection operators defined, for any $\bm{n} \in \Z^d$, by	$\mathcal{P}_{\bm{n}} = \prod_{i=1}^{d}{(\E_{\bm{n}} - \E_{\bm{n}-\bm{e}_i})}$, and for every $\omega \in \Omega$, we denote by $\P^{\omega}$ a regular version of the conditional probability given $\mathcal{F}_{\bm{0}}$, that is,	$\P^{\omega} = \P(\cdot|\mathcal{F}_{\bm{0}})(\omega)$.
	\\
	\\
	Finally, we introduce the sum that we will be studying, for every $\bm{n} \in (\mathbb{N}^*)^d$,
	\[
	S_{\bm{{n}}} = \sum_{\bm{i} = \bm{1}}^{\bm{n}}X_{\bm{i}} := \sum_{\bm{1} \le \bm{i} \le \bm{n}}X_{\bm{i}},	
	\]
	and we also set 
	\[\bar{S}_{\bm{n}} = S_{\bm{{n}}} - R_{\bm{{n}}} \textrm{\quad with \quad } R_{\bm{{n}}} = \sum_{i=1}^{d}(-1)^{i-1}\sum_{1 \le j_1 < \cdots < j_i \le d}\E_{\bm{n}^{(j_1,\cdots,j_i)}}[S_{\bm{n}}],\]
	where $\bm{n}^{(j_1,\cdots,j_d)}$  is the multi-index obtained by replacing  with $0$ all the $j_1,\cdots,j_i$-th coordinates of the multi-index $\bm{n}$ and leaving the rest unchanged.
	\\
	\\
	In dimension $d=1$, this reduces to the following expression:
	\[
	\bar{S}_{n} = S_n - \mathbb{E}[S_n|\mathcal{F}_0], \quad \textrm{for } n \in \mathbb{N}^*.
	\]
	This case was investigated by \cite{MR3083921} and therefore, we will always consider $d > 1$ in the rest of the paper. In dimension $d = 2$, the definition of $\bar{S}_{\bm{n}}$ reduces down to
	\[
	\bar{S}_{n,m} = S_{n,m} - \mathbb{E}[S_{n,m}|\mathcal{F}_{n,0}] - \mathbb{E}[S_{n,m}|\mathcal{F}_{0,m}] + \mathbb{E}[S_{n,m}|\mathcal{F}_{0,0}], \quad \textrm{for } (n,m) \in (\mathbb{N}^*)^2.
	\]
	\subsection{Functional CLT over cubic regions}
	Here we present the quenched functional CLT over cubic regions of $\mathbb{Z}^d$. These results expand Theorem 4.1, the second part of Corollary 4.3, and Theorem 4.4 (a) obtained by \cite{MR4166203} to the functional framework. It is also possible to view these results as an extension to higher dimensions of Theorem 1 established by \cite{MR3083921}. As noted by \cite{MR4166203}, the proofs of these theorems essentially reduce down to particular cases of the proofs of the functional central limit theorems over rectangular regions of $\mathbb{Z}^d$. The differences in the proofs between the two frameworks will be specified in greater detail in Section \ref{sec:proofs}. 
	\begin{theorem}\label{main result square}
		Assume that $(X_{\bm{n}})_{\bm{n} \in \Z^d}$ is defined by \eqref{Xdef} and that the filtration $(\mathcal{F}_{\bm{n}})_{\bm{n} \in \Z^d}$ given by \eqref{Fdef} is commuting. Also, assume that one of the transformations $T_i, 1 \le i \le d,$ is ergodic and that
		\begin{equation}
			\sum_{\bm{u}\geq \bm{0}}{\|\mathcal{P}_{\bm{0}}(X_{\bm{u}})\|}_{2}<\infty. \label{hannan_2}%
		\end{equation}
		Then, for $\P$-almost all $\omega \in \Omega$,
		\[
		\biggl(\frac{1}{n^{d/2}}\bar{S}_{[ n\bm{t}]}\biggr)_{\bm{t} \in [0,1]^d}  \converge{n}{\infty}{\textrm{$\mathcal{D}$}} (\sigma W_{\bm{t}})_{\bm{t} \in [0,1]^d} \quad \textrm{under }\quad \mathbb{P}^{\omega},
		\]
		where $\sigma ^2 := {\mathbb{E}[D_{\bm{0}}^2]}$ {with $D_{\bm{0}} = \sum_{\bm{i} \ge \bm{0}}\mathcal{P}_{\bm{0}}(X_{\bm{i}})$}, $(W_{\bm{t}})_{\bm{t} \in [0,1]^d}$ is a standard Brownian sheet, $[k\bm{t}] := ([kt_1], \cdots, [kt_d])$ for $k \in \mathbb{Z}$ and the convergence happens in the Skorokhod space $D([0,1]^d)$ endowed with the uniform topology. Moreover, $\sigma ^2 = \lim\limits_{n \to \infty}{\frac{\E[\bar{S}_{n, \ldots,n}^2]}{n^d}}$.
	\end{theorem}
	%$k\bm{t} := (kt_1, \ldots, kt_d)$
	In Theorem~\ref{main result square}, the random centering $R_{ [n\bm{t}]}$ cannot be avoided without additional hypotheses. As a matter of fact, for $d=1,$ \cite{MR2731055} constructed an example showing that the CLT for partial sums needs not be quenched. It should also be noticed that, for a stationary ortho-martingale, the existence of a finite second moment is not enough for the validity of a quenched CLT when the summation is taken over rectangles (see  \citealp{MR4125956}). That being said, the following corollary gives a sufficient condition to get rid of the stochastic centering $R_{\bm{n}}$ in the previous theorem. 
	\begin{corollary}\label{Cor_FCLT_d_square}
		Assume that the hypotheses of Theorem \ref{main result square} are satisfied and assume in addition that for every $i \in \{1, \ldots, d\}$, it holds
		\[
		\frac{1}{n^{d}}\mathbb{E}_{\bm{0}}\Bigl[\max _{\bm{1} \le \bm{m} \le n\bm{1}}\bigl(\mathbb{E}_{\bm{m}^{(i)}}[S_{\bm{m}}]\bigr)^2\Bigr] \converge{n}{\infty}{\normalfont{a.s.}}0
		\]
		where we recall that $\bm{m}^{(i)}$ is the multi-index obtained by replacing  with $0$ the $i$-th coordinate of the multi-index $\bm{m}$ and leaving the rest unchanged. Then, for almost all $\omega \in \Omega$,
		\begin{equation}\label{conv coro no centering square}
			\biggl(\frac{1}{n^{d/2}}S_{[ n\bm{t}]}\biggr)_{\bm{t} \in [0,1]^d}  \converge{n}{\infty}{\textrm{$\mathcal{D}$}} (\sigma W_{\bm{t}})_{\bm{t} \in [0,1]^d} \quad \textrm{under }\quad \mathbb{P}^{\omega},
		\end{equation}
		where $(W_{\bm{t}})_{\bm{t} \in [0,1]^d}$ is a standard Brownian sheet and the convergence happens in the Skorokhod space $D([0,1]^d)$ endowed with the uniform topology.
	\end{corollary}
	To end this section, we give a condition that is easier to verify but still guarantees that the convergence \eqref{conv coro no centering square} holds.
	\begin{corollary}\label{main result square no centering}
		Assume that $(X_{\bm{n}})_{\bm{n} \in \Z^d}$ is defined by \eqref{Xdef}, that $(\mathcal{F}_{\bm{n}})_{\bm{n} \in \Z^d}$ is given by \eqref{Fdef} and is a commuting filtration, and that one of the transformations $T_i, 1 \leq i \le d,$ is ergodic. If the following condition is satisfied:
		\begin{equation}
			\sum_{\bm{u}\geq \bm{1}}\frac{{\|\E_{\bm{1}}(X_{\bm{u}})\|}_{2}}{|\bm{u}|^{\frac{1}{2}}}<\infty.
			\label{higher moment 2}%
		\end{equation}
		Then, for almost all $\omega\in\Omega$, the conclusion of Corollary \ref{Cor_FCLT_d_square} holds.
	\end{corollary}
	Once again we note that this result is an extension of Corollary 2 in \cite{MR3083921} to random fields and an extension of Theorem 2.6 (a) found in \cite{MR4166203} to the functional framework. 
	\subsection{Functional CLT over rectangular regions}
	In order to obtain a functional CLT when we sum over rectangular regions, a stronger projective condition than \eqref{hannan_2} is necessary. Indeed, \cite{MR4125956} gave a counterexample to a quenched CLT over rectangles for some stationary ortho-martingale under condition \eqref{hannan_2}. This leads us to consider a projective condition in an Orlicz space associated with a specific Young function.
	\\
	\\
	Following the work of \cite{MR0126722}, we define the Luxemburg norm associated with the Young function $\Phi : [0,\infty) \to [0,\infty)$ as 
	\[
	{\|f\|}_{\Phi} = \inf \Bigr\{t > 0 : \mathbb{E}\bigl[\Phi(|f|/t)\bigr] \le 1\Bigr\}.
	\]
	In everything that follows, we will consider the Young function $\Phi_d : [0,\infty) \to [0,\infty)$ defined for every $x \in [0,\infty)$ by
	\begin{equation}\label{def_Phi}
		\Phi_d(x)=x^{2}(\log(1+x))^{d-1}.
	\end{equation}
	\begin{theorem}\label{main result}
		Assume that $(X_{\bm{n}})_{\bm{n} \in \Z^d}$ is defined by \eqref{Xdef} and that the filtration $(\mathcal{F}_{\bm{n}})_{\bm{n} \in \Z^d}$ given by \eqref{Fdef} is commuting. Also, assume that one of the transformations $T_i, 1 \le i \le d,$ is ergodic and that
		\begin{equation}
			\sum_{\bm{u}\geq \bm{0}}{\|\mathcal{P}_{\bm{0}}(X_{\bm{u}})\|}_{\Phi_d}<\infty. \label{hannan}%
		\end{equation}
		Then, for $\P$-almost all $\omega \in \Omega$,
		\begin{equation*}
			\biggl(\frac{1}{\sqrt{|\bm{n}|}}\bar{S}_{[\bm{t}\bm{n}]}\biggl)_{\bm{t} \in [0,1]^d}  \converge{\bm{n}}{\infty}{\textrm{$\mathcal{D}$}} (\sigma W_{\bm{t}})_{\bm{t} \in [0,1]^d} \quad \textrm{under }\quad \P^{\omega},
		\end{equation*}
		where $[\bm{t}\bm{n}]:= ([t_1n_1], \cdots, [t_dn_d])$, $\sigma^2$ is defined in Theorem \ref{main result square}, $(W_{\bm{t}})_{\bm{t} \in [0,1]^d}$ is a Brownian sheet, and the convergence happens in the Skorokhod space $D([0,1]^d)$. In addition, $\sigma ^2 = \lim\limits_{\bm{n} \to \infty}{\frac{\E[\bar{S}_{\bm{n}}^2]}{|\bm{n}|}}$.
	\end{theorem}
	%is the integer part of the vector $\,\bm{t}\bm{n} := (t_1n_1, \ldots, t_dn_d)$
	We remark that this result and the following ones extend Theorem 4.2, the first part of Corollary 4.3, and Theorem 2.6 (b) in \cite{MR4166203} by obtaining functional versions of these theorems.
	\begin{corollary}\label{Cor_FCLT_d}
		Suppose that the hypotheses of Theorem \ref{main result} hold and assume that in addition, for every $i\in \llbracket 1,d\rrbracket$,
		\[
		\frac{1}{|\bm{n}|}\E_{\bm{0}}\biggl[\max _{\bm{1} \le \bm{m} \le \bm{n}}\bigr(\E_{\bm{m}^{(i)}}[S_{\bm{m}}]\bigl)^2\biggr] \converge{\bm{n}}{\infty}{\normalfont{a.s.}}0.
		\]
		Then, for $\P$-almost all $\omega \in \Omega$,
		\begin{equation}\label{QFCTL no centering}
			\biggl(\frac{1}{\sqrt{|\bm{n}|}}S_{[\bm{t}\bm{n}]}\biggr)_{\bm{t} \in [0,1]^d}  \converge{\bm{n}}{\infty}{\textrm{$\mathcal{D}$}} (\sigma W_{\bm{t}})_{\bm{t} \in [0,1]^d} \quad \textrm{under }\quad \P^{\omega},
		\end{equation}
		where $(W_{\bm{t}})_{\bm{t} \in [0,1]^d}$ is a Brownian sheet and the convergence happens in the Skorokhod space $D([0,1]^d)$.
	\end{corollary}
	\begin{corollary} \label{main result no centering}
		Assume that the hypotheses of Theorem \ref{main result} and \eqref{higher moment 2} hold.	Then for almost all $\omega\in\Omega$, \eqref{QFCTL no centering} holds.
	\end{corollary}
	This last Corollary not only extends Theorem 4.4 (b) in \cite{MR4166203} to the functional case but also reduces the required condition even in the classical CLT case.
	\begin{corollary}\label{coro no centering}
		Assume that $(X_{\bm{n}})_{\bm{n} \in \Z^d}$ is defined by \eqref{Xdef}, that $(\mathcal{F}_{\bm{n}})_{\bm{n} \in \Z^d}$ is given by \eqref{Fdef} and is a commuting filtration, and that one of the transformations $T_i, 1 \leq i \le d,$ is ergodic. If the following condition is satisfied:
		\begin{equation}
			\sum_{\bm{u}\geq \bm{1}}\frac{{\|{\mathbb{E}_{\bm{1}}[X_{\bm{u}}]}\|}_{\Phi_d}}{\Phi_d^{-1}(|\bm{u}|)}<\infty.
			\label{higher moment phi}%
		\end{equation}
		Then, for almost all $\omega\in\Omega$, the conclusion of Corollary \ref{Cor_FCLT_d} holds.
	\end{corollary}
	\section{Proofs of the results}
	\label{sec:proofs}
	Before we prove the previous results, we start by defining some additional notations: 
	\begin{itemize}
		\item if $h:\Omega \to \mathbb{R}$ is a measurable function, we will denote by $h_{\bm{u}}, \bm{u} \in \mathbb{Z}^d$, the function $h \circ T^{\bm{u}}$;\
		\item for any $\bm{n} \in (\mathbb{N^*})^d$ and for any measurable function $h : \Omega \to \mathbb{R}$, we denote
		\[
		S_{\bm{{n}}}(h) = \sum_{\bm{1} \le \bm{i} \le \bm{n} }h_{\bm{i}}\quad \textrm{and}\quad  \bar{S}_{\bm{n}}(h) = S_{\bm{{n}}}(h) - R_{\bm{{n}}}(h)
		\]
		where $$R_{\bm{{n}}}(h) = \sum_{i=1}^{d}(-1)^{i-1}\sum_{1 \le j_1 < \cdots < j_i \le d}\mathbb{E}_{\bm{n}^{(j_1,\cdots,j_i)}}[S_{\bm{n}}(h)],$$ and $\bm{n}^{(j_1,\cdots,j_i)}$  is the multi-index whose $j_1,\cdots,j_i-$th coordinates are $0$ and the others are equal to the corresponding coordinates of $\bm{n}$;\
		\item for any $i \in \llbracket 1,d\rrbracket$ and for any $\ell \in \mathbb{N}$, we denote
		\[
		\mathcal{F}_{\ell}^{(i)} = \bigvee_{\substack{\bm{k} \in \mathbb{Z}^d\\ k_i \le l}} \mathcal{F}_{\bm{k}};
		\]\
		\item we set $L^2\log^{d-1}L(\mathcal{G})$ to be the set of $\mathcal{G}$-measurable functions $h :\Omega \to \mathbb{R}$ such that \\$\mathbb{E}\Bigl[h^2\max\bigl(0,\log |h|\bigr)^{d-1}\Bigr] < \infty$; if $\mathcal{G} = \mathcal{F}$, we simply write $L^2\log^{d-1}L(\mathcal{F}) = L^2\log^{d-1}L$;\
		\item if $h \in L^2\log^{d-1}L$, then we define the maximal operator $h^* = \sup_{\bm{m} > \bm{0}}{\frac{1}{|\bm{m}|}\sum_{\bm{1} \le \bm{i} \le \bm{m}}|h|\circ T^{\bm{i}}}$.
	\end{itemize}
	Let us start with the proof of Theorem~\ref{main result} as it is the most general result. Moreover, the computations used in the proof of Theorem~\ref{main result square} are a particular case of the computations used in the proof of Theorem~\ref{main result} and will be largely skipped.
	\\
	\\
	The proof of Theorem~\ref{main result} relies on the following important lemma which we will refer to as the Main Lemma in the rest of the paper
	\begin{lemma}[Main Lemma]
		\label{lem:main}
		For any $\mathcal{F}_{\bm{0}}$-measurable function $h \in L^2\log ^{d-1}L$ satisfying the following condition:
		\begin{equation}\label{hannan_h}
			\sum_{\bm{u} \ge \bm{0}}{\|\mathcal{P}_{\bm{0}}(h_{\bm{u}})\|}_{\Phi_d} < \infty,
		\end{equation} there exists an integrable function $g$ such that for all $\bm{N} \in (\mathbb{N}^*)^d$,
		\[
		\sqrt{\mathbb{E}_{\bm{0}}\biggl[\max_{\bm{1} \le \bm{n} \le \bm{N}}\frac{1}{|\bm{n}|}\bigl|\overline{S}_{\bm{n}}(h)\bigr|^2\biggr]} \le g \quad \mathbb{P}-\mathrm{a.s.}
		\]
	\end{lemma}
	To establish this lemma, we shall first obtain the following intermediary lemma.
	\begin{lemma}\label{lemme_int}
		For any function $h \in L^2\log ^{d-1}L$, there exists a constant $C>0$ such that for all $\bm{u} \in \mathbb{Z}^d$, we have
		\[
		{\Bigg\|\sqrt{\biggl(\bigl|\mathcal{P}_{\bm{0}}(h_{\bm{u}})\bigr|^2\biggr)^*}\Biggr\|}_1 \le C{\|\mathcal{P}_{\bm{0}}(h_{\bm{u}})\|}_{\Phi_d}.
		\]
	\end{lemma}
	\begin{proof}[Proof of Lemma \ref{lemme_int}]
		Let $h \in L^2\log ^{d-1}L$, $\bm{u} \in \mathbb{Z}^d$ and $t > {\|\mathcal{P}_{\bm{0}}(h_{\bm{u}})\|}_{\Phi_d}$. We let 
		\[
		\Omega _{t} = \Bigl\{ \omega \in \Omega : 4\bigl(\mathcal{P}_{\bm{0}}(h_{\bm{u}})\bigr)^2(\omega) >t^2\Bigr\}. 
		\]
		According to Corollary 1.7 of Chapter 6 in \cite{MR0797411}, there exists a constant $C_d > 0$ such that
		\begin{align*}
			\mathbb{P}\Biggl(\sup_{\bm{n} \in (\mathbb{N}^*)^d}\frac{1}{|\bm{n}|}\sum_{\bm{1}\le\bm{i} \le\bm{n}}\Bigl(\mathcal{P}_{\bm{0}}(h_{\bm{u}})\circ T^{\bm{i}}\Bigr)^2 > t^2\Biggr) & \le 
			C_d\int _{\Omega _{t}}\frac{4\bigl(\mathcal{P}_{\bm{0}}(h_{\bm{u}})\bigr)^2}{t^2}\Biggl(\log\Biggl( \frac{4\bigl(\mathcal{P}_{\bm{0}}(h_{\bm{u}})\bigr)^2}{t^2}\Biggr)\Biggr)^{d-1}\mathrm{d}\mathbb{P}\\
			& \le 2^{d-1}C_d\int _{\Omega _{t}}\frac{\bigl(2\mathcal{P}_{\bm{0}}(h_{\bm{u}})\bigr)^2}{t^2}\Biggl(\log \Biggl(1+\frac{2|\mathcal{P}_{\bm{0}}(h_{\bm{u}})|}{t}\Biggr)\Biggr)^{d-1}\mathrm{d}\mathbb{P}\\
			& \le 2^{d+1}C_dt^{-2}{\|{\mathcal{P}_{\bm{0}}(h_{\bm{u}})}\|}_{\Phi_d}^2.
		\end{align*}
		The last inequality results from the fact that
		\[
		{\|\mathcal{P}_{\bm{0}}(h_{\bm{u}})\|}_{\Phi_d} = \inf \biggl\{ t > 0 : \mathbb{E}\biggl[\Phi_d\biggl(\frac{|\mathcal{P}_{\bm{0}}(h_{\bm{u}})|}{t}\biggr)\biggr] \le 1\biggr\}.
		\]
		Indeed, by letting $t_0 = {\|\mathcal{P}_{\bm{0}}(h_{\bm{u}})\|}_{\Phi_d}$, we have
		\[
		\mathbb{E}\Biggl[(\mathcal{P}_{\bm{0}}\bigl(h_{\bm{u}})\bigr)^2 \biggl(\log \biggl(1 + \frac{|\mathcal{P}_{\bm{0}}(h_{\bm{u}})|}{t_0}\biggr)\biggr)^{d-1}\Biggr] \le t_0^2.
		\]
		Hence since $t > t_0$,
		\[
		\int_{\Omega}\frac{\bigl(\mathcal{P}_{\bm{0}}(h_{\bm{u}})\bigr)^2}{t^2}\Biggl(\log \Biggl(1+\frac{|\mathcal{P}_{\bm{0}}(h_{\bm{u}})|}{t}\Biggr)\Biggr)^{d-1}\mathrm{d}\mathbb{P} \le t_0^2t^{-2}.
		\]
		Therefore, applying this inequality to $h' = 2h \in L^2\log^{d-1}L$, we get
		\[
		\int _{\Omega _{t}}\frac{\bigl(2\mathcal{P}_{\bm{0}}(h_{\bm{u}})\bigr)^2}{t^2}\Biggl(\log \Biggl(1+\frac{2|\mathcal{P}_{\bm{0}}(h_{\bm{u}})|}{t}\Biggr)\Biggr)^{d-1}\mathrm{d}\mathbb{P} \le 4t^{-2}{\|\mathcal{P}_{\bm{0}}(h_{\bm{u}})\|}_{\Phi_d}^2.
		\]
		Thus
		\begin{align*}
			{\Biggl\|\sqrt{\biggl(\bigl|\mathcal{P}_{\bm{0}}(h_{\bm{u}})\bigr|^2\biggr)^*}\Biggr\|}_1 & = \int_{0}^{\infty}{\mathbb{P}\Biggl(\sup_{\bm{n} \in (\mathbb{N}^*)^d}\frac{1}{|\bm{n}|}\sum_{\bm{1}\le\bm{i} \le\bm{n}}\Bigl(\mathcal{P}_{\bm{0}}(h_{\bm{u}})\circ T^{\bm{i}}\Bigr)^2 > t^2\Biggr)\mathrm{d}t}\\
			& \le \int_{0}^{t_0}1\mathrm{d}t + \int_{t_0}^{\infty}{\mathbb{P}\Biggl(\sup_{\bm{n} \in (\mathbb{N}^*)^d}\frac{1}{|\bm{n}|}\sum_{\bm{1}\le\bm{i} \le\bm{n}}\Bigl(\mathcal{P}_{\bm{0}}(h_{\bm{u}})\circ T^{\bm{i}}\Bigr)^2 > t^2\Biggr)\mathrm{d}t}\\
			& \le (2^{d+1}C_d+1){\|\mathcal{P}_{\bm{0}}(h_{\bm{u}})\|}_{\Phi_d}.
		\end{align*}
	\end{proof}
	\begin{proof} [Proof of the Main Lemma]
		We consider a measurable function $h$ satisfying the hypotheses of the lemma and we let $\bm{n}, \bm{N}\in(\mathbb{N}^*)^d$ such that $\bm{n}\le\bm{N}$. Then, we start by studying the quantity $\bar{S}_{\bm{n}}(h)$ using the following projective decomposition  (see \citealp{MR3798239}):
		
		\begin{equation*}
			S_{\bm{n}}(h)-R_{\bm{n}}(h)=\sum_{\bm{1} \le \bm{i} \le \bm{n}}{\mathcal{P}}_{\bm{i}}\Biggl({\sum
				\limits_{\bm{i} \le \bm{u} \le \bm{n}}}h_{\bm{u}}\Biggr)
			=\sum_{\bm{1} \le \bm{i} \le \bm{n}}{\mathcal{P}}_{\bm{0}}\Biggl({\sum
				\limits_{\bm{0} \le \bm{u} \le \bm{n}-\bm{i}}}h_{\bm{u}}\Biggr) \circ T^{\bm{i}}.\label{ort dec_2}%
		\end{equation*}
		By exchanging the sums, we get
		\begin{equation*}
			\bar{S}_{\bm{n}}(h)=\sum_{\bm{0} \le \bm{u} \le \bm{n} - \bm{1}}\sum_{\bm{1} \le \bm{i} \le \bm{n} - \bm{u}}\mathcal{P}_{\bm{0}}(h_{\bm{u}})\circ T^{\bm{i}}.
		\end{equation*}
		
		Then, recalling that $\bm{n} \le \bm{N}$, we obtain
		
		\[
		\bigl|\bar{S}_{\bm{n}}(h)\bigr|\leq \sum_{\bm{0} \le \bm{u} \le \bm{N} - \bm{1}}\max_{\bm{1}\leq \bm{k}\leq \bm{N}}\Biggl \lvert\sum_{\bm{1} \le \bm{i} \le \bm{k}}\mathcal{P}_{\bm{0}}(h_{\bm{u}})\circ T^{\bm{i}}\Biggr \rvert.
		\]
		
		Note that for all $\bm{u} \ge \bm{0}$, the partial sum
		$\displaystyle \sum_{\bm{1} \le \bm{i} \le \bm{k}}\mathcal{P}_{\bm{0}}(h_{\bm{u}})\circ T^{\bm{i}}$ is an ortho-martingale. Using Cairoli's inequality for ortho-martingales (see \citealp{MR1914748}), we find that
		
		\[
		\mathbb{E}_{\bm{0}}\Biggl[\max_{\bm{1}\leq \bm{k}\leq \bm{N}} \Biggl|\sum_{\bm{1} \le \bm{i}\le \bm{k}}\mathcal{P}_{\bm{0}}(h_{\bm{u}})\circ T^{\bm{i}} \Biggr|^2 \Biggr]%
		\leq 
		2^{2d} \mathbb{E}_{\bm{0}}\Biggl[  \Biggl(\sum_{\bm{1} \le \bm{i} \le \bm{N}}\mathcal{P}_{\bm{0}}(h_{\bm{u}})\circ T^{\bm{i}}\Biggr)^2 \Biggr].
		\]
		
		By orthogonality, we obtain for all $\bm{N} \in (\mathbb{N}^*)^d$
		
		\[
		\sqrt{
			\mathbb{E}_{\bm{0}}\biggl[\max_{\bm{1}\leq \bm{n}\leq \bm{N}}\bigl|\bar{S}_{\bm{n}}(h)\bigr|^2\biggr]}%
		\leq %
		2^d \sum_{\bm{u} \ge \bm{0}}  \sqrt{\sum_{\bm{1} \le \bm{i} \le \bm{N}}\mathbb{E}_{\bm{0}}\Bigl[\bigr(\mathcal{P}_{\bm{0}}(h_{\bm{u}})\bigl)^2\circ T^{\bm{i}} \Bigr]} 
		\leq 2^{d} \sqrt{|\bm{N}|} \sum_{\bm{u} \ge \bm{0}}\sqrt{\Bigl(\bigl|\mathcal{P}_{\bm{0}}(h_{\bm{u}})\bigr|^2\Bigr)^{*}}.
		\]
		Since the previous inequality is satisfied for all $\bm{N} \in (\mathbb{N}^*)^d$, then it also holds
		\begin{equation}\label{lemme_princip}
			\sqrt{\mathbb{E}_{\bm{0}}\biggl[\max_{\bm{1} \le \bm{n} \le \bm{N}}\frac{1}{|\bm{n}|}\bigl|\overline{S}_{\bm{n}}(h)\bigr|^2\biggr]} \le 2^d\sum_{\bm{u} \ge \bm{0}}\sqrt{\Bigl(\bigl|\mathcal{P}_{\bm{0}}(h_{\bm{u}})\bigr|^2\Bigr)^*}.
		\end{equation}
		However, according to Lemma \ref{lemme_int} and hypothesis \eqref{hannan_h}, there exists $C > 0$ such that
		\[
		{\Biggl\| \sum_{\bm{u} \ge \bm{0}}\sqrt{\Bigl(\bigl|\mathcal{P}_{\bm{0}}(h_{\bm{u}})\bigr|^2\Bigr)^*}\Biggr\|}_1 \le C\sum_{\bm{u} \ge \bm{0}}{\|\mathcal{P}_{\bm{0}}(h_{\bm{u}})\|}_{\Phi_d} < \infty.
		\]
		This concludes the proof of the main lemma.
	\end{proof}
	\begin{proof}[Proof of Theorem \ref{main result}]
		For any $n \in \mathbb{N}^*$, we let
		\[
		X_{\bm{0}}^{(n)} = \sum_{\bm{j} \in \llbracket -n,0\rrbracket^d}\mathcal{P}_{\bm{j}}(X_{\bm{0}})
		\]
		Given the regularity of $X_{\bm{0}}$, the sequence of random variable  $(X_{\bm{0}}-X_{\bm{0}}^{(n)})_{n \in \mathbb{N}}$ converges almost surely to $0$ and using \eqref{lemme_princip}, we get the inequality
		\begin{equation*}
			\limsup\limits_{\bm{N} \to \infty}\sqrt{\mathbb{E}_{\bm{0}}\biggl[\max_{\bm{1} \le \bm{m} \le \bm{N}}\frac{1}{|\bm{m}|}\bigl|\overline{S}_{\bm{m}}(X_{\bm{0}}-X_{\bm{0}}^{(n)})\bigr|^2\biggr]} \le 2^d\sum_{\bm{u}\ge \bm{0}}\sqrt{\biggl(\Bigl|\mathcal{P}_{\bm{0}}\Bigl((X_{\bm{0}} - X_{\bm{0}}^{(n)}) \circ T^{\bm{u}}\Bigr)\Bigr|^2\biggr)^*}
		\end{equation*}
		for all $n \in \mathbb{N}^*$. Then, using lemma \ref{lemme_int}, there exists a constant $C$ such that
		\[
		{\Biggl\|\sqrt{\limsup\limits_{\bm{N} \to \infty}\mathbb{E}_{\bm{0}}\biggl[\max_{\bm{1} \le \bm{m} \le \bm{N}}\frac{1}{|\bm{m}|}\bigl|\overline{S}_{\bm{m}}(X_{\bm{0}}-X_{\bm{0}}^{(n)})\bigr|^2\biggr]}\Biggr\|}_1 \le C\sum_{\bm{u} \ge \bm{0}}{{\Bigl\|\mathcal{P}_{\bm{0}}\Bigl((X_{\bm{0}} - X_{\bm{0}}^{(n)}) \circ T^{\bm{u}}\Bigr)\Bigr\|}_{\Phi_d}} \converge{n}{\infty}{}0.
		\]
		Therefore, there exists an increasing sequence of integers $(n_k)_{k \in \mathbb{N}}$ such that
		\begin{equation}\label{approx_fn}
			\lim\limits_{k \to \infty}\limsup\limits_{\bm{N} \to \infty}\mathbb{E}_{\bm{0}}\biggl[\max_{\bm{1} \le \bm{m} \le \bm{N}}\frac{1}{|\bm{m}|}\bigl|\overline{S}_{\bm{m}}(X_{\bm{0}}-X_{\bm{0}}^{(n_k)})\bigr|^2\biggr] = 0 \ \ \ \ \mathrm{a.s.}
		\end{equation}
		Moreover, we also have, for all $n \in \mathbb{N}^*$
		\begin{equation}\label{conv_rest}
			\frac{1}{|\bm{N}|}\mathbb{E}_{\bm{0}}\biggl[\max_{\bm{1} \le \bm{i} \le \bm{N}}\bigl|R_{\bm{i}}(X_{\bm{0}}^{(n)})\bigr|^2\biggr] \converge{\bm{N}}{\infty}{\textrm{a.s.}}0.
		\end{equation}
		Indeed, using the triangle inequality, it is enough to show that for all $\bm{i} \in \llbracket -n,0\rrbracket^d$ 
		\[
		\frac{1}{|\bm{N}|}\mathbb{E}_{\bm{0}}\biggl[\max_{\bm{1} \le \bm{j} \le \bm{N}}\bigl|R_{\bm{j}}(\mathcal{P}_{\bm{i}}(X_{\bm{0}}))\bigr|^2\biggr] \converge{\bm{N}}{\infty}{\textrm{a.s.}}0.
		\]
		This holds true by applying the following lemma.
		\begin{lemma}\label{fact4}
			For any square integrable $\mathcal{F}_{\bm{0}}$-measurable function $h$, the condition 
			\begin{equation}
				\label{higher_moment_2_h}
				\sum_{\bm{u} \ge \bm{1}}\frac{{\|\mathbb{E}_{\bm{1}}[h_{\bm{u}}]\|}_{2}}{|\bm{u}|^{1/2}} < \infty
			\end{equation} implies
			\[
			\frac{1}{|\bm{N}|}\mathbb{E}_{\bm{0}}\biggl[\max_{\bm{1} \le \bm{n} \le \bm{N}} \bigl|R_{\bm{n}}(h)\bigr|^2\biggr] \converge{\bm{N}}{\infty}{\textrm{a.s.}} 0.
			\]
		\end{lemma}
		We delay the proof of this lemma to later in this section.
		\\
		\\
		Remark that the proof of Proposition 4.1 in \cite{MR3264437} can be easily adapted to the case of Orlicz spaces; so that for some fixed $n \in \mathbb{N}^*$, we get the following martingale-coboundary decomposition
		\[                                                                          
		X_{\bm{0}}^{(n)} = \sum_{S \subset \llbracket 1,d \rrbracket}h_S^{(n)} \circ \prod_{j \in S^c}(I-T_j),
		\]
		where $h_S^{(n)} \in \bigcap_{i\in S}{\Bigl(L^2\log^{d-1}L\bigl(\mathcal{F}_{0}^{(i)}\bigr) \ominus L^2\log^{d-1}L\bigl(\mathcal{F}_{-1}^{(i)}\bigr)\Bigr)}$ for all $S \subset \llbracket 1,d \rrbracket$ and   using the convention $\prod _{j \in \emptyset}{(I-T_j)} = I$.
		Moreover 
		\[
		h_{\llbracket  1,d\rrbracket}^{(n)} = \sum_{\bm{i} \in \mathbb{Z}^d}{\mathcal{P}_{\bm{0}}\bigl(X_{\bm{0}}^{(n)} \circ T^{\bm{i}}\bigr)}.
		\]
		According to the proof of Remark 11 in \cite{MR4125956} (see also the proof of Theorem 7 in the same article), the  following almost-sure convergence
		\begin{equation}\label{approx_mart}
			\mathbb{P}^{\omega}\biggl(\max_{\bm{1} \le \bm{m} \le \bm{N}}\frac{1}{|\bm{N}|}\Bigl|S_{\bm{m}}\bigl(X_{\bm{0}}^{(n)} - d_n\bigr)\Bigr|^2 \ge \epsilon\biggr) \converge{\bm{N}}{\infty}{\textrm{a.s.}} 0
		\end{equation}
		holds for all $\epsilon > 0$, where $d_n = h_{\llbracket  1,d\rrbracket}^{(n)}$.
		Moreover, letting $\bm{N} \in (\mathbb{N}^*)^d$ and $D_0 = \sum_{\bm{i} \in \mathbb{Z}^d}\mathcal{P}_{\bm{0}}(X_{\bm{i}})$, we get
		\[
		S_{\bm{N}}(D_0 - d_n) = \sum_{\bm{1} \le \bm{i} \le \bm{N}}\bigl(D_{\bm{i}} -D_{{\bm{i}}}^{(n)}\bigr),
		\]
		where
		\[
		D_{\bm{i}}= \sum_{\bm{j} \in \mathbb{Z}^d}{\mathcal{P}_{\bm{i}}\bigl(X_{\bm{i} - \bm{j}}\bigr)} \quad \mathrm{and}\quad D_{\bm{i}}^{(n)} = \sum_{\bm{j} \in \llbracket -n,0\rrbracket^d}{\mathcal{P}_{\bm{i}}\bigl(X_{\bm{i} - \bm{j}}\bigr)}.
		\] 
		Hence, given that $\bigl(D_{\bm{i}} - D_{\bm{i}}^{(n)}\bigr)_{\bm{i} \in \Z^d}$ is an ortho-martingale difference field and according to Cairoli's inequality, we have
		\[
		\mathbb{P}^{\omega}\Biggl(\frac{1}{\sqrt{|\bm{N}|}}\max_{\bm{1} \le \bm{i} \le \bm{N}}\bigl|S_{\bm{i}}(D_0 - d_n)\bigr| > \epsilon\Biggr)  \le \frac{2^{2d}}{\epsilon^2|\bm{N}|}\sum_{\bm{1} \le \bm{i} \le \bm{N}}\mathbb{E}_{\bm{0}}\biggl[\Bigl(D_{\bm{i}} - D_{\bm{i}}^{(n)}\Bigr)^2\biggr].
		\]
		Let us note that
		\[
		\sqrt{\frac{1}{|\bm{N}|}\sum_{\bm{1} \le \bm{i} \le \bm{N}}\mathbb{E}_{\bm{0}}\biggl[\Bigl(D_{\bm{i}} - D_{\bm{i}}^{(n)}\Bigr)^2\biggr]} \le \sum_{\bm{j} \not \in \llbracket-n,0\rrbracket^d}\sqrt{\frac{1}{|\bm{N}|}\sum_{\bm{1} \le \bm{i} \le \bm{N}}\mathbb{E}_{\bm{0}}\Bigl[\bigl(\mathcal{P}_{\bm{0}}(X_{-\bm{j}})\bigr)^{2}\circ T^{\bm{i}}\Bigr]}.
		\]
		According to the ergodic Theorem 1.1 of Chapter 6 in \cite{MR0797411} for Dunford Schwartz operators and Lemma 7.1 in \cite{MR3224292}, we have the convergence
		\[
		\lim\limits_{\bm{N} \to \infty}\frac{1}{|\bm{N}|}\sum_{\bm{1} \le \bm{i} \le \bm{N}}\mathbb{E}_{\bm{0}}\Bigl[\bigl(\mathcal{P}_{\bm{0}}(X_{-\bm{j}})\bigr)^{2}\circ T^{\bm{i}}\Bigr] = \mathbb{E}\Bigl[\bigl(\mathcal{P}_{\bm{j}}(X_{\bm{0}})\bigr)^{2}\Bigr]\ \ \ \ \mathrm{a.s.}
		\]
		for all $\bm{j} \not \in\llbracket-n,0\rrbracket^d$. Since $\sum_{\bm{j} \ge \bm{0}}{\|\mathcal{P}_{\bm{0}}(X_{\bm{j}})\|}_2 < \infty$, we get
		\begin{equation}\label{approx_mn}
			\lim\limits_{n \to \infty}\lim\limits_{\bm{N} \to \infty}\frac{2^{2d}}{\epsilon^2|\bm{N}|}\sum_{\bm{1} \le \bm{i} \le \bm{N}}\mathbb{E}_{\bm{0}}\biggl[\Bigl(D_{\bm{i}} - D_{\bm{i}}^{(n)}\Bigr)^2\biggr] = 0 \ \ \ \ \mathrm{a.s.}
		\end{equation}
		\\
		\\
		Combining \eqref{approx_fn}, \eqref{conv_rest},\eqref{approx_mart} and \eqref{approx_mn}, we obtain that for all $\epsilon > 0$,
		\[
		\limsup\limits_{\bm{N} \to \infty} \mathbb{P}^{\omega}\Biggl(\frac{1}{\sqrt{|\bm{N}|}}\max_{\bm{1} \le \bm{m} \le \bm{N}}\bigl|\bar{S}_{\bm{m}} - S_{\bm{m}}(D_0)\bigr| \ge \epsilon\Biggr)  = 0\ \ \ \ \mathrm{a.s.}
		\]
		We conclude by noticing that the field $(D_0 \circ T^{\bm{i}})_{\bm{i} \in \mathbb{Z}^d}$ satisfies a functional central limit theorem (according to Theorem 10 in \citealp{MR4125956}) and therefore the expected result is obtained by applying Theorem 3.1 in \cite{MR0293706}.
	\end{proof}
	
	\begin{proof}[Proof of the Theorem~\ref{main result square}]
		The proof of this theorem is very similar to the previous one, with the exception of using Theorem 2.8 instead of Theorem 1.1 of Chapter 6 in \cite{MR0797411} and Lemma 1.4 in the same Chapter (applied to the abstract maximal operator $Mf := \sup_{n \in \mathbb{N}}{\frac{1}{n^d}\sum_{\bm{1} \le \bm{i} \le n\bm{1}}|f| \circ T^{\bm{i}}}$, see Definition 1.3 of Chapter 6 and Corollary 2.2 of Chapter 1 in \citealp{MR0797411}) instead of Corollary 1.7 in order to obtain the $L^2$ versions of lemma \ref{lemme_int} and the \hyperref[lem:main]{Main Lemma} mentioned below.
	\end{proof}
	\begin{lemma}[$L^2$ version of the Main Lemma \ref{lem:main}]
		For any function $h \in L^2(\mathcal{F}_{\bm{0}})$ satisfying the following condition:
		\begin{equation}\label{hannan_2'}
			\sum_{\bm{u} \ge \bm{0}}{\|\mathcal{P}_{\bm{0}}(h_{\bm{u}})\|}_{2} < \infty,
		\end{equation} there exist an integrable function $g$ such that for all $N \in \mathbb{N}^*$,
		\[
		\sqrt{\mathbb{E}_{\bm{0}}\biggl[\max_{1 \le n \le N}\frac{1}{n^{d}}\bigl|\overline{S}_{n\bm{1}}(h)\bigr|^2\biggr]} \le g \quad \mathbb{P}-\mathrm{a.s}.
		\]
	\end{lemma}
	\begin{lemma}[$L^2$ version of Lemma \ref{lemme_int}]
		For all functions $h \in L^2$, there exists a constant $C>0$ such that for all $\bm{u} \in \mathbb{Z}^d$, we have
		\[
		{\Biggl\|\sqrt{\Bigl(\bigl|\mathcal{P}_{\bm{0}}(h_{\bm{u}})\bigr|^2\Bigr)^{\star}}\Biggr\| }_1 \le C{\|\mathcal{P}_{\bm{0}}(h_{\bm{u}})\|}_{2},
		\]
		where $h^{\star} = \sup_{n \in \mathbb{N}^*}{\frac{1}{n^d}}\sum_{\bm{1} \le \bm{i} \le n\bm{1}}|h|\circ T^{\bm{i}}$.
	\end{lemma}
	The following proof of Corollary \ref{Cor_FCLT_d} can be easily adapted to obtain Corollary \ref{Cor_FCLT_d_square} by using the Theorem \ref{main result square} instead of Theorem \ref{main result}.
	\\
	
	\begin{proof}[Proof of Corollary \ref{Cor_FCLT_d}]
		According to Theorem \ref{main result} and Theorem 3.1 in \cite{MR0293706}, it is enough to show
		\[
		\frac{1}{|\bm{n}|}\mathbb{E}_{\bm{0}}\biggl[\max _{\bm{1} \le \bm{m} \le \bm{n}}R_{\bm{m}}^2\biggr] \converge{\bm{n}}{\infty}{\textrm{a.s.}} 0.
		\]
		Let $\bm{m} \in \mathbb{Z}^d$, and recall that
		\[
		R_{\bm{m}} = \sum_{i=1}^{d}(-1)^{i-1}\sum_{1 \le j_1 < \cdots < j_i \le d}\mathbb{E}_{{\bm{m}}^{(j_1,\cdots,j_i)}}[S_{\bm{m}}],
		\]
		where $\bm{m}^{(j_1,\ldots, j_i)}$ is the multi-index such that the $j_k$-th, $1 \le k \le i$ coordinates are zero and the others are equal to the corresponding coordinates of $\bm{m}$.
		\\
		\\
		Using the triangle inequality, it is enough to prove that the property
		\begin{equation}\label{HR}\tag*{\textrm{$P(j_1, \ldots, j_{i})$}}
			\frac{1}{|\bm{n}|}\mathbb{E}_{\bm{0}}\biggl[\max _{\bm{1} \le \bm{m} \le \bm{n}}\bigl(\mathbb{E}_{{\bm{m}}^{(j_1,\cdots,j_i)}}[S_{\bm{m}}]\bigr)^2\biggr] \converge{\bm{n}}{\infty}{\textrm{a.s.}} 0
		\end{equation}
		holds for any $j_1 < \cdots < j_i, 1 \le i \le d$. We establish that via induction on $i$.
		\\
		\\
		The terms satisfying $i=1$ have this property according to the hypothesis of the corollary. By induction, the corollary will be proven if we can show that if property \ref{HR} is verified for all $j_1 < \cdots < j_i$ for some $i<d$, then $P(j_1, \ldots, j_{i+1})$ also holds for all $j_1 < \cdots < j_{i+1}$. For the sake of simplicity and without loss of generality, we will only establish that 
		\[
		\forall j \in \llbracket1,d\rrbracket, P(j) \quad \Longrightarrow \quad P(1,2).
		\]
		In other words, we use the hypothesis of the corollary to show
		\[
		\frac{1}{|\bm{n}|}\mathbb{E}_{\bm{0}}\biggl[\max _{\bm{1} \le \bm{m} \le \bm{n}}\bigl(\mathbb{E}_{\bm{m}^{(1,2)}}[S_{\bm{m}}]\bigr)^2\biggr] \converge{\bm{n}}{\infty}{\textrm{a.s.}} 0.
		\]
		According to Jensen's inequality, we have
		\[
		\mathbb{E}_{\bm{0}}\biggl[\max _{\bm{1} \le \bm{m} \le \bm{n}}\bigl(\mathbb{E}_{\bm{m}^{(1,2)}}[S_{\bm{m}}]\bigr)^2\biggr] = \mathbb{E}_{\bm{0}}\biggl[\max _{\bm{1} \le \bm{m} \le \bm{n}}\bigl(\mathbb{E}_{\bm{m}^{(1,2)}}\bigl[\mathbb{E}_{\bm{m}^{(1)}}[S_{\bm{m}}]\bigr]\bigr)^2\biggr] \le \mathbb{E}_{\bm{0}}\biggl[\max _{\bm{1} \le \bm{m} \le \bm{n}}\bigl(\mathbb{E}_{\bm{m}^{(1)}}[S_{\bm{m}}]\bigr)^2\biggr].
		\]
		So
		\[
		\frac{1}{|\bm{n}|}\mathbb{E}_{\bm{0}}\biggl[\max _{\bm{1} \le \bm{m} \le \bm{n}}\bigl(\mathbb{E}_{\bm{m}^{(1,2)}}[S_{\bm{m}}]\bigr)^2\biggr] \le \frac{1}{|\bm{n}|}\mathbb{E}_{\bm{0}}\biggl[\max _{\bm{1} \le \bm{m} \le \bm{n}}\bigl(\mathbb{E}_{\bm{m}^{(1)}}[S_{\bm{m}}]\bigr)^2\biggr] \converge{\bm{n}}{\infty}{\textrm{a.s.}} 0.
		\]
	\end{proof}
	
	Before continuing with the proof, we establish Lemma \ref{fact4}.

	\begin{proof} [Proof of Lemma~\ref{fact4}]
		We show that for any $i \in \llbracket1,d\rrbracket$ and any $1 \le j_1 < \cdots < j_i \le d$, we have the convergence 
		\[
		\frac{1}{|\bm{N}|}\mathbb{E}_{\bm{0}}\biggl[\max_{\bm{1} \le \bm{n} \le \bm{N}}\bigl(\mathbb{E}_{\bm{n}^{(j_1, \ldots, j_i)}}[S_{\bm{n}}(h)]\bigr)^2\biggr] \converge{\bm{N}}{\infty}{\textrm{a.s.}}0.
		\]
		In order to do so, we use an induction on $k = d-i$ with $d$ fixed. For $k = 0$ and in the same way as in the proof of Lemma 3.2 in \cite{MR4166203} (see also the proof of Theorem 4.4 in the same article), we establish that
		\[
		\frac{\bigl(\mathbb{E}_{\bm{0}}[S_{\bm{n}}(h)]\bigr)^2}{|\bm{n}|} \converge{\bm{n}}{\infty}{\textrm{a.s.}} 0.
		\]
		Hence
		\[
		\frac{1}{|\bm{N}|} \max _{\bm{1} \le \bm{n} \le \bm{N}} \bigl(\mathbb{E}_{\bm{0}}[S_{\bm{n}}(h)]\bigr)^2 \converge{\bm{N}}{\infty}{\textrm{a.s.}} 0.
		\]
		Now, we suppose that the desired property holds for some $k-1 < d-1$. Without loss of generality, we establish the property only for $(j_1, \ldots, j_{k}) = ( 1,\ldots, k)$; it is enough to show that
		\[
		\frac{1}{|\bm{N}|}\mathbb{E}_{\bm{0}}\biggl[\max_{\bm{1} \le \bm{n} \le \bm{N}}\bigl(  \mathbb{E}_{\bm{n}^{(j_1, \ldots, j_{k})}}[S_{\bm{n}}(h)]-\mathbb{E}_{\bm{n}^{(j_2, \ldots, j_{k})}}[S_{\bm{n}}(h)]\bigr)^{2}\biggr]\converge{\bm{N}}{\infty}{\textrm{a.s.}}0.
		\]
		Let $\bm{n}, \bm{N} \in (\mathbb{N}^*)^d$ such that $\bm{n} \le \bm{N}$, then the following decomposition holds 
		\[
		\mathbb{E}_{\bm{0}}\biggl[\max_{\bm{1} \le \bm{n} \le \bm{N}}\bigl(  \mathbb{E}_{\bm{n}^{(j_1, \ldots, j_{k})}}[S_{\bm{n}}(h)]-\mathbb{E}_{\bm{n}^{(j_2, \ldots, j_{k})}}[S_{\bm{n}}(h)]\bigr)^{2}\biggr] = \mathbb{E}_{\bm{0}}\Biggl[\max_{\bm{1} \le \bm{n} \le \bm{N}}\Biggl( \sum_{i=1}^{n_{1}}P_{\bm{n}^{(j_1, \ldots, j_k)},i}^{(1)}(S_{\bm{n}}(h)) \Biggr)^{2}\Biggr],
		\]
		where $P_{\bm{n}^{(j_1, \ldots, j_k)}, i}^{(1)}(S_{\bm{n}}(h)) = \mathbb{E}_{i\bm{e}_{1} + \bm{n}^{(j_1, \ldots, j_k)}}[S_{\bm{n}}(h)] - \mathbb{E}_{(i-1)\bm{e}_{1}+ \bm{n}^{(j_1, \ldots, j_k)}}[S_{\bm{n}}(h)]$ and $\bm{e}_{1}$ is the multi-index whose coordinates are all zero except for the first one which is $1$.
		\\
		\\
		Since $h$ is $\mathcal{F}_{\bm{0}}$-measurable, we have
		\[
		\Biggl|\sum_{i=1}^{n_{1}}P_{\bm{n}^{(j_1, \ldots, j_k)},i}^{(1)}(S_{\bm{n}}(h))\Biggr| \le \sum_{\bm{1} \le \bm{u} \le \bm{N}}\max_{1\leq k\leq N_{1}} \Biggl|\sum_{i=1}^{k}P_{\bm{u}^{(j_1, \ldots, j_k)},i}^{(1)}(h_{\bm{u}})\Biggr|.
		\]
		Therefore, according to Doob's inequality for martingales, it follows
		\[
		\sqrt{\mathbb{E}_{\bm{0}}\Biggl[\max_{\bm{1} \le \bm{n} \le \bm{N}}\Biggl( \sum_{i=1}^{\bm{n}_{1}}P_{\bm{n}^{(j_1, \ldots, j_k)},i}^{(1)}(S_{\bm{n}}(h)) \Biggr)^{2}\Biggr]} \le 2\sum_{\bm{1} \le \bm{u} \le \bm{N}}\sqrt{\mathbb{E}_{\bm{0}}\Biggl[\Biggl(\sum_{i=1}^{N_1}P_{\bm{u}^{(j_1, \ldots, j_k)},i}^{(1)}(h_{\bm{u}})\Biggr)^2\Biggr]}.
		\]
		Let $c > 0$, we use the following decomposition
		\[
		\frac{1}{\sqrt{|\bm{N}|}}\sum_{\bm{1} \le \bm{u} \le \bm{N}}\sqrt{\mathbb{E}_{\bm{0}}\Biggl[\Biggl(\sum_{i=1}^{N_1}P_{\bm{u}^{(j_1, \ldots, j_k)},i}^{(1)}(h_{\bm{u}})\Biggr)^2\Biggr]} =: \mathrm{I}_{\bm{N},c} + \mathrm{II}_{\bm{N},c}
		\]
		where 
		\[
		\mathrm{I}_{\bm{N},c} = \frac{1}{\sqrt{|\bm{N}|}}\sum_{1 \le u_1 \le N_1}\sum_{1 \le u_2,\ldots u_d \le c}\sqrt{\mathbb{E}_{\bm{0}}\Biggl[\Biggl(\sum_{i=1}^{N_1}P_{\bm{u}^{(j_1, \ldots, j_k)},i}^{(1)}(h_{\bm{u}})\Biggr)^2\Biggr]}
		\]
		and $\mathrm{II}_{\bm{N},c}$ is the remainder of the initial sum.
		\\
		\\
		Let us show that
		\[
		\limsup\limits_{\bm{N} \to \infty} \mathrm{I}_{\bm{N},c} = 0 \ \ \ \ \mathrm{a.s.}
		\] 
		Indeed, by an orthogonality argument
		\[
		\mathrm{I}_{\bm{N},c} \le \frac{c^{d-1}}{\sqrt{|\bm{N}|}}\sup_{1 \le u_2,\ldots, u_d \le c}\sum_{u_1 \ge 0}\sqrt{\sum_{i=1}^{N_1}\mathbb{E}_{\bm{0}}\biggl[\Bigl(P_{\bm{u}^{(j_1, \ldots, j_k)},0}^{(1)}(h_{\bm{u}})\Bigr)^2\circ T_1^i\biggr]}.
		\]
		However, according to Lemma 7.1 in \cite{MR3224292}, we get the convergence
		\[
		\lim\limits_{N_1 \to \infty}\frac{1}{N_1}\sum_{1 \le i \le N_1}\mathbb{E}_{\bm{0}}\biggl[\Bigl(P_{{\bm{u}}^{(j_1, \ldots, j_k)},0}^{(1)}(h_{\bm{u}})\Bigr)^2 \circ T_1^i\biggl] = \mathbb{E}\biggl[\Bigl(P_{\bm{u}^{(j_1,\ldots,j_k)}, 0}^{(1)}(h_{\bm{u}})\Bigr)^2\biggm|\mathcal{I}_1\biggr] \quad \textrm{a.s.}
		\]
		where $\mathcal{I}_1$ is the invariant $\sigma$-algebra of the transformation $T_1$. Hence
		\[
		\limsup\limits_{\bm{N} \to \infty} \mathrm{I}_{\bm{N},c} = 0 \ \ \ \ \mathrm{a.s.}
		\]
		Each sum appearing in $\mathrm{II}_{\bm{N},c}$, admits at least one direction (different from the first one) for which the index is at least equal to $c+1$. Without loss of generality, we will only treat the case where the second direction has an index at least equal to $c+1$ and all other directions have the full range of indexes. Then
		\begin{align*}
			& \frac{1}{\sqrt{|\bm{N}|}}\sum_{1 \le u_1 \le N_1} \sum_{c+1 \le u_2 \le N_2}\sum_{1 \le u_3 \le N_3}\cdots\sum_{1 \le u_d \le N_d}\sqrt{\mathbb{E}_{\bm{0}}\Biggl[\Biggl(\sum_{i=1}^{N_1}P_{\bm{u}^{(j_1, \ldots, j_k)},i}^{(1)}(h_{\bm{u}})\Biggr)^2\Biggr]}\\
			& \qquad \qquad \le \sum_{u_2 \ge c+1}\sum_{u_1 \ge 0} \sum_{u_3 \ge 1}\cdots\sum_{u_d \ge 1}(u_2\cdots u_d)^{-\frac{1}{2}}\sqrt{\frac{1}{N_1}\sum_{i=1}^{N_1}\mathbb{E}_{\bm{0}}\biggl[\Bigl(P_{\bm{u}^{(j_1,\ldots,j_k)}, 0}^{(1)}(h_{\bm{u}})\Bigr)^2\circ T^i_1\biggr]}.
		\end{align*}
		Once again, applying Lemma 7.1 in \cite{MR3224292}, we obtain
		\[
		\limsup\limits_{\bm{N}\to \infty}\mathrm{II}_{\bm{N},c} \le \sum_{u_2 \ge c+1} \sum_{u_1 \ge 0}\sum_{u_2 \ge 1}\cdots\sum_{u_d \ge 1}\frac{\sqrt{\mathbb{E}\biggl[\Bigl(P_{\bm{u}^{(j_1,\ldots,j_k)}, 0}^{(1)}(h_{\bm{u}})\Bigr)^2\biggm|\mathcal{I}_1\biggr]}}{\sqrt{u_2\cdots u_d}}.
		\]
		Since \eqref{higher_moment_2_h}  implies (45) in \cite{MR4166203} (see the proof of Theorem 4.4), we get
		\[
		\lim\limits_{c \to\infty}\limsup\limits_{\bm{N} \to \infty} \mathrm{II}_{\bm{N},c} = 0\ \ \ \ \mathrm{a.s.}
		\]
		This concludes the proof.
	\end{proof}
	
	Corollaries~\ref{main result no centering} and~\ref{main result square no centering} are direct consequences of this lemma and the previous results.
	
	\begin{proof}[Proof of Corollary~\ref{main result no centering}]
		This theorem is a consequence of Lemma \ref{fact4} and Theorem \ref{main result}. 
	\end{proof}
	
	\begin{proof}[Proof of Corollary~\ref{main result square no centering}]
		The theorem is a consequence of Lemma \ref{fact4}, Theorem \ref{main result} and Lemma 3.3 in  \cite{MR4166203} (for $d > 2$, see Theorem 4.4 (a) and its proof in \citealp{MR4166203}) and Theorem \ref{main result square}.
	\end{proof}
	The rest of this section will be dedicated to proving Corollary~\ref{coro no centering}. We start by making a few remarks concerning the Luxemburg norms. Let us note that for $x\ge0$ and $0<\lambda\leq \mathrm{e}-1$,
	\begin{equation*}
		\log\Bigl(1+\frac{x}{\lambda}\Bigr)\log(1+\lambda)\leq \log(1+x).
	\end{equation*}
	
	Recall that the function $\Phi_d : [0,\infty) \to [0,\infty)$ is defined by 
	\begin{equation*}
		\Phi_d (x)=x^{2}(\log(1+x))^{d-1}
	\end{equation*} for all $x \in [0,\infty)$. Then, we deduce the
	following remarkable property of the function $\Phi_d$.\newline
	For $x>0$ and $0<\lambda\leq \mathrm{e}-1$,
	\begin{equation}\label{prop1 varphi}
		\Phi_d \Bigl(\frac{x}{\lambda}\Bigr)=\Bigl(\frac{x}{\lambda}\Bigr)^{2}\Bigl(\log\Bigl(1+\frac{x}{\lambda}\Bigr)\Bigr)^{d-1}\leq \frac{x^2(\log(1+x))^{d-1}}{\lambda^2(\log(1+\lambda))^{d-1}}= \frac{\Phi_d(x)}{\Phi_d (\lambda)}.
	\end{equation}
	
	Besides, since $\Phi_d$ is a convex function, we also have 
	
	\begin{equation}\label{prop2 varphi}
		\Phi_d \Bigl(\frac{x}{\lambda}\Bigr)=\Phi_d \Bigl(\frac{x}{\lambda}+\Bigl(1-\frac{1}{\lambda}\Bigr)\cdot 0\Bigr)\leq \frac{\Phi_d(x)}{\lambda}+\Bigl(1-\frac{1}{\lambda}\Bigr)\Phi_d(0)= \frac{\Phi_d(x)}{\lambda},
	\end{equation}
	for  $x\ge0$ and $\lambda\ge 1$.
	\\
	\\
	Obviously, the function $\Phi_d$ defined by \eqref{def_Phi} is bijective and we denote by $\Phi_d^{-1}$ its inverse function. The following lemma might be well-known but we could not find it in the
	literature. 
	
	\begin{lemma}\label{Lem tool} Let $X \in L^2\log^{d-1}L$. If $ \Phi_d ^{-1}\bigl(\mathbb{E}\bigl[\Phi_d (|X|)\bigr]\bigr)\leq \mathrm{e}-1$, then
		\begin{equation*}
			{\|X\|}_{\Phi_d }\leq \Phi_d ^{-1}\bigl(\mathbb{E}\bigl[\Phi_d (|X|)\bigr]\bigr),
		\end{equation*}%
		and if $\mathbb{E}\bigl[\Phi_d (|X|)\bigr]\ge1$, then 
		\begin{equation*}
			{\|X\|}_{\Phi_d }\leq \mathbb{E}\bigl[\Phi_d (|X|)\bigr].
		\end{equation*}%
	\end{lemma}
	\begin{proof}[Proof of Lemma~\ref{Lem tool}]
		If $X = 0$ almost surely, then the property is evident. Else, suppose that $\mathbb{P}(X = 0) \not = 1$, and recall the definition of Luxembourg norm 
		\begin{equation*}
			{\|X\|}_{\Phi_d }=\inf \biggl\{\lambda>0 : \mathbb{E}\biggl[\Phi_d \biggl(\frac{|X|}{\lambda }\biggr)\biggr]\leq
			1\biggr\}
		\end{equation*}%
		Note that by the properties of $\Phi_d $ for any $0<\lambda \leq \mathrm{e}-1$ 
		\begin{equation*}
			\mathbb{E}\biggl[\Phi_d \biggl(\frac{|X|}{\lambda }\biggr)\biggr]\leq \frac{\mathbb{E}\bigl[\Phi_d (|X|)\bigr]}{\Phi_d (\lambda )}%
			.
		\end{equation*}%
		From this inequality it follows that if $\lambda $ is the solution to the
		equation $\mathbb{E}\bigl[\Phi_d (|X|)\bigr]=\Phi_d(\lambda)$, we have necessarily that $%
		\mathbb{E}\Bigl[\Phi_d \Bigl(\frac{|X|}{\lambda }\Bigr)\Bigr]\leq 1,$ and then ${\|X\|}_{\Phi_d }\leq\lambda = \Phi_d ^{-1}\bigl(\mathbb{E}\bigl[\Phi_d (|X|)\bigr]\bigr).$
		\\
		\\
		For the case $\mathbb{E}\bigl[\Phi_d (|X|)\bigr]> 1$, the proof is similar using property \eqref{prop2 varphi} of $\Phi_d$.
	\end{proof}

	\begin{lemma}
		\label{fact5}
		Condition~\eqref{higher moment phi} implies  $ \sum_{\bm{u}\geq\bm{0}}\lVert \mathcal{P}_{\bm{0}}(X_{\bm{u}})\rVert_{\Phi_d}<\infty$.
	\end{lemma}
	
	\begin{proof}[Proof of Lemma~\ref{fact5}]
		Let $\bm{a}, \bm{b} \in \mathbb{Z}^d$ such that $\bm{a} \le \bm{b}$.  Denote by $\Psi_d$ the conjugate function associated with $\Phi_d$ defined in the following way
		\[
		\Psi_d(x)=\sup_{y\geq 0}(xy-\Phi_d(y))
		\]
		for $x \ge 0$.
		By the generalized Holder inequality for Orlicz spaces (see \citealp{MR1113700}, p.58), we have 
		\begin{align}
			\label{eq:holder}
			\sum_{\bm{u}\geq \bm{1}}{\| \mathcal{P}_{\bm{0}}(X_{\bm{u}})\|}_{\Phi_d} = & 
			\sum_{\bm{n} \geq \bm{0}}\sum_{\bm{v} = 2^{\bm{n}}}^{2^{\bm{n}+\bm{1}}-\bm{1}}{\|\mathcal{P}_{\bm{0}}(X_{\bm{v}})\|}_{\Phi_d}\nonumber\\
			\leq & 2\sum_{\bm{n}\geq \bm{0}}\inf \Biggl\{ \eta > 0 : \sum_{\bm{v} = 2^{\bm{n}}}^{2^{\bm{n}+\bm{1}}-\bm{1}}\Psi_d \biggl(\frac{1}{\eta }%
			\biggr)\leq 1\Biggr\} \cdot\inf \Biggl\{\eta > 0 : \sum_{\bm{v} = 2^{\bm{n}}}^{2^{\bm{n}+\bm{1}}-\bm{1}}\Phi_d \biggl(\frac{\|{\mathcal{P}_{\bm{0}}(X_{\bm{v}})\|}_{\Phi_d}}{\eta} \biggr)\leq 1\Biggr\},
		\end{align}
		where $2^{\bm{n}} = (2^{n_1}, \ldots, 2^{n_d})$. Computing the second term in the sum, we get%
		\begin{equation*}
			\inf \Biggl\{ \eta > 0 : \sum_{\bm{v} = 2^{\bm{n}}}^{2^{\bm{n}+\bm{1}}-\bm{1}}\Psi_d \biggl(\frac{1}{\eta }%
			\biggr)\leq 1\Biggr\}=%
			\frac{1}{\Psi_d ^{-1}(|2^{\bm{n}}|^{-1})}.
		\end{equation*}%
		In order to control the second term in \eqref{eq:holder}, we make the following remark: if $f$ is a nondecreasing convex function then for any $\eta > 0$, the following holds
		\begin{align*}
			f\Biggl({\Biggl\|\sum_{\bm{v} = 2^{\bm{n}}}^{2^{\bm{n}+\bm{1}}-\bm{1}}\frac{\mathcal{P}_{-\bm{v}}(X_{\bm{0}})}{\eta}\Biggr\|}_{\Phi_d}\Bigg) & = f\Biggl({\Biggr\|\frac{1}{\eta}\sum_{i=0}^d(-1)^{d-i}\sum_{1 \le j_1 < \cdots < j_i \le d}\mathbb{E}_{-2^{\bm{n}+\bm{1}^{(j_1, \ldots,j_i)}}}[X_{\bm{0}}]\Biggr\|}_{\Phi_d}\Biggr)\\
			& \le \frac{1}{2^d}\sum_{i=0}^d\sum_{1 \le j_1 < \cdots < j_i \le d}f\Biggl({\Biggl\|\frac{2^d}{\eta}(-1)^{d-i}\mathbb{E}_{-2^{\bm{n}+\bm{1}^{(j_1, \ldots, j_i)}}}[X_{\bm{0}}]\Biggr\|}_{\Phi_d}\Biggr)\\
			& \le f\Biggl(\frac{2^d}{\eta}{\|\mathbb{E}_{-2^{\bm{n}}}[X_{\bm{0}}]\|}_{\Phi_d}\Biggr).
		\end{align*}
		Note that in the previous inequalities, the term corresponding to $i = 0$ is by convention $(-1)^d\mathbb{E}_{-2^{\bm{n}+\bm{1}}}[X_{\bm{0}}]$. It is enough to control the second term in \eqref{eq:holder} only when $\mathbb{E}\Bigl[\Phi_d \Bigl(\frac{|\mathcal{P}_{-\bm{v}}(X_{\bm{0}})|}{\eta }\Bigr)\Bigr]\leq \Phi_d(\mathrm{e}-1)$ for any $\bm{v} \ge \bm{0}$, as the other cases can be proved using similar arguments and are left to the reader.	By Lemma \ref{Lem tool} above, if $\mathbb{E}\Bigl[\Phi_d \Bigl(\frac{|\mathcal{P}_{-\bm{v}}(X_{\bm{0}})|}{ \eta }\Bigr)\Bigr]\leq \Phi_d(\mathrm{e}-1)$ for any $\bm{v} \ge \bm{0}$,
		we also have%
		\begin{equation*}
			\Phi_d \biggl(\frac{{\|\mathcal{P}_{-\bm{v}}(X_{\bm{0}})\|}_{\Phi_d }}{\eta }\biggr)\leq
			\mathbb{E}\biggl[\Phi_d \biggl(\frac{|\mathcal{P}_{-\bm{v}}(X_{\bm{0}})|}{\eta }\biggr)\biggr].
		\end{equation*}%
		
		So%
		\begin{equation}
			\label{ineq_norm_esperance}
			\sum_{\bm{v} = 2^{\bm{n}}}^{2^{\bm{n}+\bm{1}}-\bm{1}}\Phi_d \biggl(\frac{{\|\mathcal{P}_{-\bm{v}}(X_{\bm{0}})\|}_{\Phi_d }}{\eta }\biggr)\leq \sum_{\bm{v} = 2^{\bm{n}}}^{2^{\bm{n}+\bm{1}}-\bm{1}}\mathbb{E}\biggl[\Phi_d \biggl(\frac{|\mathcal{P}_{-\bm{v}}(X_{\bm{0}})|}{\eta }\biggr)\biggr].
		\end{equation}
		Before proceeding, we state the following lemma which is a version of the Rosenthal inequality in the Orlicz space associated with $\Phi_d$. Its proof will be given in the appendix.
		\begin{lemma}
			\label{rosenthal_ineq_orlicz}
			If $(d_{\bm{u}})_{\bm{u} \in (\mathbb{N})^*}$ is an ortho-martingale difference field and $0<\epsilon <1$, then there exists two constants $C_1,C_2 > 1$, which only depend on $d$ and $\epsilon$, such that
			\begin{equation*}
				\sum_{\bm{u} = \bm{0}}^{\bm{n} - \bm{1}}\mathbb{E}\bigl[\Phi_d(|d_{\bm{u}}|)\bigr] \le  C_1\max\Biggl\{\varphi_d^{-1}\Biggl(C_2{\Biggl\|\sum_{\bm{u} = \bm{0}}^{\bm{n} - \bm{1}}d_{\bm{u}}\Biggr\|}_{\Phi_d}\Biggr), \phi_d\circ f_{\epsilon}\Biggl(C_2{\Biggl\|\sum_{\bm{u} = \bm{0}}^{\bm{n} - \bm{1}}d_{\bm{u}}\Biggr\|}_{\Phi_d}\Biggr)\Biggr\},
			\end{equation*}
			with $\phi_d(x) = x^{d+1}\bigl(\log(1+x)\bigr)^{d-1}$, $\varphi_d(x) = x\bigl(\log(1+x)\bigr)^{d-1}$ and $f_{\epsilon}(x) = x^{1/(1-\epsilon)}$ for all $x \ge 0$. 
		\end{lemma}
		We let $\epsilon = \frac{1}{d+2}$ and we notice that while the function $\phi_d\circ f_{\epsilon}$ is convex on $[0,\infty)$, that is not the case for $\varphi_d^{-1}$. To solve this issue, remark that $\Bigl(\varphi_d'\Bigl(\frac{C_1^{-1}}{2}\Bigr)\Bigr)^{-1}x+\frac{C_1^{-1}}{2} \ge \varphi_d^{-1}(x)$ for all $x \ge 0$ and thus the function $\rho:[0,\infty)\to [0,\infty)$ defined by $\rho(x) = \max\biggl\{\Bigl(\varphi_d'\Bigl(\frac{C_1^{-1}}{2}\Bigr)\Bigr)^{-1}x+\frac{C_1^{-1}}{2},\phi_d\circ f_{\epsilon} (x)\biggr\}$ is convex and greater than $\varphi_d^{-1}$. Applying Lemma \ref{rosenthal_ineq_orlicz}, we can show that there exists $C_1,C_2 > 1$ such that for any $\eta > 0$,
		\begin{align*}
			\hspace{-0.5cm}\sum_{\bm{v} = 2^{\bm{n}}}^{2^{\bm{n}+\bm{1}}-\bm{1}}\Phi_d\biggl(\frac{{\|\mathcal{P}_{\bm{0}}(X_{\bm{v}})\|}_{\Phi_d}}{\eta}\biggr) &\leq \sum_{\bm{v} = 2^{\bm{n}}}^{2^{\bm{n}+\bm{1}}-\bm{1}}\mathbb{E}\biggl[\Phi_d \biggl(\frac{|\mathcal{P}_{-\bm{v}}(X_{\bm{0}})|}{\eta }\biggr)\biggr]\\
			&\le C_1\max\Bigg\{\varphi_d^{-1}\Biggl(C_2{\Biggl\|\sum_{\bm{v} = 2^{\bm{n}}}^{2^{\bm{n}+\bm{1}}-\bm{1}}\frac{\mathcal{P}_{-\bm{v}}(X_{\bm{0}})}{\eta }\Biggr\|}_{\Phi_d}\Biggr), \phi_d\circ f_{\epsilon}\Biggl(C_2{\Biggl\|\sum_{\bm{v} = 2^{\bm{n}}}^{2^{\bm{n}+\bm{1}}-\bm{1}}\frac{\mathcal{P}_{-\bm{v}}(X_{\bm{0}})}{\eta }\Biggr\|}_{\Phi_d}\Biggr)\Bigg\}\\
			& \le C_1\rho\Biggl(C_2{\Biggl\|\sum_{\bm{v} = 2^{\bm{n}}}^{2^{\bm{n}+\bm{1}}-\bm{1}}\frac{\mathcal{P}_{-\bm{v}}(X_{\bm{0}})}{\eta }\Biggr\|}_{\Phi_d}\Biggr).
		\end{align*}
		According to the remark above, we get that
		\[
		\rho\Biggl(C_2{\Biggl\|\sum_{\bm{v} = 2^{\bm{n}}}^{2^{\bm{n}+\bm{1}}-\bm{1}}\frac{\mathcal{P}_{-\bm{v}}(X_{\bm{0}})}{\eta }\Biggr\|}_{\Phi_d}\Biggr) \le \rho\biggl(\frac{2^dC_2}{\eta}{\|\mathbb{E}_{-2^{\bm{n}}}[X_{\bm{0}}]\|}_{\Phi_d}\biggr). 
		\]
		Since $\rho$ is continous and $\rho(0) < C_1^{-1}$, the set $\bigl\{ \mu > 0 : \rho(\mu) \le C_1^{-1}\bigr\}$ is non-empty. Moreover
		\begin{align*}
			C_1\rho\biggl(C_2\frac{2^d}{\eta}{\|\mathbb{E}_{-2^{\bm{n}}}[X_{\bm{0}}]\|}_{\Phi_d}\biggr) \le 1 
			& \quad \Longleftrightarrow \quad \eta \ge \frac{2^dC_2}{\rho^{-1}(C_1^{-1})}{\|\mathbb{E}_{-2^{\bm{n}}}[X_{\bm{0}}]\|}_{\Phi_d}.
		\end{align*}
		Therefore, setting $C = \frac{2^dC_2}{\rho^{-1}(C_1^{-1})}$, we conclude that
		\begin{align*}
			\inf \Biggl\{ \eta > 0 : \sum_{\bm{v} = 2^{\bm{n}}}^{2^{\bm{n}+\bm{1}}-\bm{1}}\Phi_d \biggl(\frac{{\|\mathcal{P}_{\bm{0}}(X_{\bm{v}})\|}_{\Phi_d }}{\eta} \biggr)\leq 1\Biggr\} & \leq \inf \biggl\{\eta > 0 : C_1\rho\biggl(C_2\frac{2^d}{\eta}{\|\mathbb{E}_{-2^{\bm{n}}}[X_{\bm{0}}]\|}_{\Phi_d}\biggr)\leq 1\biggr\}\\
			& = C{\|\mathbb{E}_{-2^{\bm{n}}}[X_{\bm{0}}]\|}_{\Phi_d }.
		\end{align*}
		Since ${\|\mathbb{E}_{-{\bm{n}}}[X_{\bm{0}}]\|}_{\Phi_d }$ is nonincreasing in all directions of $\bm{n}$, we obtain that for any $\bm{n}$ such that $n_k > 0$ for all $k \in \llbracket 1,d\rrbracket$, we have
		\[
		\frac{|2^{\bm{n}}|}{{\Phi_d}^{-1}(|2^{\bm{n}}|)}{\|\mathbb{E}_{-2^{\bm{n}}}[X_{\bm{0}}]\|}
		_{\Phi_d}\leq 2^d\sum_{\bm{u}=2^{\bm{n}-\bm{1}}}^{2^{\bm{n}}-\bm{1}}\frac{{\|
				\mathbb{E}_{\bm{0}}[X_{\bm{u}}]\|}_{\Phi_d}}{{\Phi_d}^{-1}(|\bm{u}|)}.
		\]
		So, for some positive constant $K$,%
		\begin{equation*}
			\sum_{\bm{v}\geq \bm{1}}{\|\mathcal{P}_{\bm{0}}(X_{\bm{u}})\|}_{\Phi_d }\leq K\sum_{\bm{n}\geq \bm{1}}%
			\frac{\Phi_d^{-1}(|2^{\bm{n}}|)}{|2^{\bm{n}}|{\Psi_d}^{-1}(|2^{\bm{n}}|^{-1})}\sum_{\bm{u}=2^{\bm{n}-\bm{1}}}^{2^{\bm{n}}-\bm{1}}\frac{{\|\mathbb{E}_{-\bm{u}}(X_{\bm{0}})\|}_{\Phi_d }}{{\Phi_d}^{-1}(|\bm{u}|)}.
		\end{equation*}
		However, there exists a constant $K' > 0$ such that ${\Phi_d}^{-1}(|2^{\bm{n}}|)\underset{{\bm{n} \to \infty}}{\sim} |2^{\bm{n}}|{\Psi_d}^{-1}(|2^{\bm{n}}|^{-1})K'$. Hence, according to the previous  inequalities, we have shown that (\ref{higher moment phi}) implies
		\[
		\sum_{\bm{u}\geq \bm{1}}\lVert\mathcal{P}_{\bm{0}}(X_{\bm{u}})\rVert_{\Phi_d}<\infty.
		\]
		In the same way, we have for every $i \in \llbracket 1,d\rrbracket$ and for every $(j_1,\ldots,j_i) \in \llbracket 1,d\rrbracket^i$ such that $j_1 < \cdots < j_i$, 
		\[
		\sum_{\bm{u}^{(j_1,\ldots,j_i)}\ge\bm{1}^{(j_1,\ldots,j_i)}}\lVert\mathcal{P}_{\bm{0}}(X_{\bm{u}^{(j_1,\ldots,j_i)}})\rVert_{\Phi_d}<\infty.
		\]
		Hence (\ref{hannan}) is fulfilled.
	\end{proof}

	\begin{proof}[Proof of Corollary~\ref{coro no centering}]
		This corollary is a consequence of Theorem \ref{main result} and Lemmas \ref{fact4} and \ref{fact5}.
	\end{proof}
	\section{Examples}
	\label{sec:examples}
	In this section, we present various examples of applications of the different results we obtained. First, we will focus on linear processes as well as a particular case of nonlinearity known as the Volterra field. In doing so, we improve on the results by \cite{MR4166203} by requiring weaker assumptions on both the moment of the innovations and the coefficients which appear in each example. More precisely, we obtain a functional CLT despite only requiring that the i.i.d. innovations belong to the Orlicz space $L^2\log^{d-1}L$ instead of the Lebesgue space $L^q$ with $q> 2$ as is required by \cite{MR4166203}. Afterward, we shall discuss the case of Hölder continuous functions of linear fields. To the best of the authors' knowledge, it does not appear that quenched central limit theorems have been derived in this context. Finally, we study the case of weakly dependent processes in the sense of \cite{Wu2005} which play an important role in many physical models such as particle systems (see \citealp{MR2108619, MR1182416}). As far as we know, this class of random fields has been seldom, if ever, investigated for quenched central limit theorems. In that, our theorems provide some innovative convergence results for these processes. 
	\\
	\par Throughout this section, we will write $a \triangleleft b$ whenever $a \le Cb$ with $C > 0$ being a constant which can only depend on some fixed parameters. Recall that the function $\Phi_d : [0,\infty) \to [0,\infty)$ is bijective and defined by \eqref{def_Phi}.
	\subsection{Linear field with independent innovations}
	The first application of our results will deal with linear fields as it presents an opportunity to show how our results improve on that of \cite{MR4166203}. It is also a very common type of fields which present a lot of interest in and of themselves. The main argument of the proof relies on Corollary \ref{main result no centering}.
	\begin{example}
		\label{linear} (Linear field) Let $(\xi_{\bm{n}})_{\bm{n}\in \Z^{d}}$
		be a random field of independent, identically distributed random variables,
		which are centered and satisfy $\E\Bigl[  |\xi_{\bm{0}}|^{2}\bigl(\log (1 + |\xi_{\bm{0}}|)\bigr)^{d-1}\Bigr]  <\infty$. For
		$\bm{k}\geq\bm{0}$ define%                                              
		\[
		X_{\bm{k}}=\sum_{\bm{j}\geq\bm{0}}a_{\bm{j}}\xi_{\bm{k}%    
			-\bm{j}},
		\]
		where $a_{\bm{u}}$ are real coefficients such that $\sum_{\bm{u}\geq\bm{0}}a_{\bm{u}}^{2}         
		<\infty$. In addition, assume that
		\begin{equation}
			\sum_{\bm{k}\geq\bm{1}}\frac{1}{\sqrt{|\bm{k}|}}\biggl(\sum
			_{\bm{j}\geq\bm{k-1}}a_{\bm{j}}^{2}\biggr)^{\frac{1}{2}}<\infty.
			\label{cond lin}%                                                               
		\end{equation}
		Then the quenched convergence \eqref{QFCTL no centering} holds.
	\end{example}
	The results obtained by \cite{MR4166203} (Remark 6.2 (c)) required the existence of $q-$th moment, with $q > 2$, of the innovation $\xi _{\bm{0}}$ to obtain the quenched functional CLT; meanwhile, we only require that $\xi _{\bm{0}}$ satisfy a weaker Orlicz condition to obtain that result. Additionally, we require weaker assumptions on the coefficients $a_{\bm{u}}, \bm{u} \in \Z^d$.
	\begin{proof}[Proof of Example~\ref{linear}]
		Let $ \bm{u} \ge \bm{0}$. According to the independence of the $\xi_{\bm{n}}$, we have
		\[
		\mathcal{P}_{\bm{0}}(X_{\bm{u}}) = a_{\bm{u}}\xi_{\bm{0}} \quad \textrm{and for $\bm{u} \ge \bm{1}$,} \quad \E_{\bm{1}}[X_{\bm{u}}]=\sum_{\bm{j} \ge \bm{u} - \bm{1}}a_{\bm{j}}\xi_{\bm{u}-\bm{j}}.
		\]
		Applying Burkholder inequality, we obtain
		\begin{align*}
			{\|\E_{\bm{1}}[X_{\bm{u}}]\|}_{2} & = {\Biggl\|\sum_{\bm{j} \ge \bm{u} - \bm{1}}a_{\bm{j}}\xi_{\bm{u}-\bm{j}}\Biggr\|}_{2}\\
			& \triangleleft \sqrt{\sum_{\bm{j} \ge \bm{u} - \bm{1}}a_{\bm{j}}^2{\|\xi_{\bm{u}-\bm{j}}\|}_{2}^2}.
		\end{align*}
		By stationarity of the random field $(\xi_{\bm{i}})_{\bm{i} \in \Z^d}$, we get that
		\[
		{\|\E_{\bm{1}}[X_{\bm{u}}]\|}_{2} \triangleleft {\|\xi_{\bm{0}}\|}_{2}\Biggl(\sum_{\bm{j} \ge \bm{u} - \bm{1}}a_{\bm{j}}^2\Biggr)^{\frac{1}{2}}.
		\]
		Thus, using assumption (\ref{cond lin}) and since ${\|\xi_{\bm{0}}\|}_{\Phi_d} < \infty$, we have shown that condition (\ref{higher moment 2}) is satisfied. Now, noticing that
		\[\sum_{\bm{u} \ge \bm{0}}{{\|\mathcal{P}_{\bm{0}}(X_{\bm{u}})\|}_{\Phi_d}} = \frac{{\|\xi_{\bm{0}}\|}_{\Phi_d}}{{\|\xi_{\bm{0}}\|}_{2}} \sum_{\bm{u} \ge \bm{0}}{{\|\mathcal{P}_{\bm{0}}(X_{\bm{u}})\|}_{2}}\]
		whenever $\mathbb{P}(\xi_{\bm{0}} = 0) < 1$ and using Lemma 3.3 in \cite{MR4166203}, we find that condition \eqref{hannan} is satisfied. To get the result, we simply apply Corollary \ref{main result no centering}.
	\end{proof}
	\subsection{Nonlinearity: the case of Volterra fields}
	As for the linear case, a lot is known about Volterra fields including some quenched limit theorem under a variety of conditions as in \cite{MR2359065, MR4166203}. Applying our results will require a bit more work than in the previous case since the lack of linearity means that we cannot guarantee that Hannan's criterion is satisfied  if we only assume that the coefficients of the Volterra field satisfy a condition similar to \eqref{cond lin}. In particular, a new method of proof relying on Corollary \ref{Cor_FCLT_d} will be required leading to slightly stronger assumptions than in Example~\ref{linear}.
	\begin{example}
		\label{Volterra}(Volterra field) Let $(\xi_{\bm{n}})_{\bm{n}\in \Z^{d}%  
		}$ be a random field of independent, identically distributed, and centered
		random variables satisfying $\E\Bigl[|\xi_{\bm{0}}|^{2}\bigl(\log (1 + |\xi_{\bm{0}}|)\bigr)^{d-1}\Bigr]  <\infty$.
		For $\bm{k}\geq\bm{0}$, define%                                         
		\[
		X_{\bm{k}}=\sum_{\bm{u},\bm{v}\geq\bm{0}}a_{\bm{u},\bm{v}}\xi_{\bm{k}-\bm{u}}\xi_{\bm{k}-\bm{v}}
		\]
		where $a_{\bm{u},\bm{v}}$ are real coefficients with $a_{\bm{u}%
			,\bm{u}}=0$ and $\sum_{\bm{u},\bm{v}\geq\bm{0}}a_{\bm{u},\bm{v}}%         
		^{2}<\infty$. In addition, assume that                                     
		\begin{equation}
			\sum_{\bm{k\geq1}}\frac{1}{{\Phi_d}^{-1}(|\bm{k}|)}\Biggl(\sum
			_{\bm{u},\bm{v}\geq \bm{k}-\bm{1}}a_{\bm{u},\bm{v}%     
			}^{2}\Biggr)^{1/2} < \infty, \label{cond volt}%                                   
		\end{equation}
		Then the quenched functional CLT in Corollary \ref{Cor_FCLT_d} holds.
	\end{example}
	This result is a generalization of the quenched functional CLT obtained in \cite{MR4166203} (Example 6.3). Here, we only require an Orlicz space type condition on the innovation $\xi_{\bm{0}}$ and we weaken the condition (54) in \cite{MR4166203} to condition (\ref{cond volt}). Additionally, we can see that \eqref{cond volt} is not a very tractable condition to work with; therefore we provide the following stronger, but easier to verify, sufficient condition for \eqref{cond volt} to hold:
	$$\sum_{\bm{k}\geq\bm{1}}\frac{\bigl(\log(|\bm{k}|)\bigr)^{\frac{d-1}{2}}}{|\bm{k}|^{\frac{1}{2}}}\Biggl(\sum
	_{\bm{u},\bm{v}\geq \bm{k}-\bm{1}}a_{\bm{u},\bm{v}%     
	}^{2}\Biggr)^{1/2} < \infty.$$
	\begin{proof}[Proof of Example \ref{Volterra}]
		Let $ \bm{k} \ge \bm{1}$  and note that%
		\[
		\E_{\bm{1}}[X_{\bm{k}}]=\sum_{\bm{u},\bm{v} \ge \bm{k} - \bm{1}}a_{\bm{u},\bm{v}}\xi _{\bm{k}-\bm{u}}\xi _{\bm{k}-\bm{v}}.
		\]
		Let $(\xi_{\bm{n}}^{\prime})_{\bm{n}\in \Z^{d}}$ and $(\xi_{\bm{n}%
		}^{\prime\prime})_{\bm{n}\in \Z^{d}}$ be two independent copies of
		$(\xi_{\bm{n}})_{\bm{n}\in \Z^{d}}$. By applying a decoupling inequality by de la Pe\~{n}a and Gin\'e (Theorem 3.1.1 in \citealp{MR1666908}, p.99), we get for any $t > 0$,
		\begin{align*}
			\E\bigl[\Phi_d\bigl(|\E_{\bm{1}}[X_{\bm{k}}]|/t\bigr)\bigr]& = \E\Biggl[\Phi_d\Biggl(\frac{1}{t}\Biggl|\sum_{\bm{u},\bm{v} \ge \bm{k} - \bm{1}}a_{\bm{u},\bm{v}}\xi _{\bm{k}-\bm{u}}\xi _{\bm{k}-\bm{v}}\Biggr|\Biggr)\Biggr]\\
			& \le \E\Biggl[\Phi_d\Biggl(\frac{C}{t}\Biggl|\sum_{\bm{u},\bm{v} \ge \bm{k} - \bm{1}}a_{\bm{u},\bm{v}}\xi _{\bm{k}-\bm{u}}^{\prime}\xi _{\bm{k}-\bm{v}}^{\prime\prime}\Biggr|\Biggr)\Biggr],
		\end{align*}
		with $C > 0$. Hence 
		\[
		{\|\E_{\bm{1}}[X_{\bm{k}}]\|}_{\Phi_d} \triangleleft {\Biggl\|\sum_{\bm{u},\bm{v} \ge \bm{k} - \bm{1}}a_{\bm{u},\bm{v}}\xi _{\bm{k}-\bm{u}}^{\prime}\xi _{\bm{k}-\bm{v}}^{\prime\prime}\Biggr\|}_{\Phi_d}.
		\]
		Therefore, using the inequality ${\|XY\|}_{\Phi_d} \triangleleft {\|X\|}_{\Phi_d}{\|Y\|}_{\Phi_d}$ for any two independent random variables $X, Y$ such that both ${\|X\|}_{\Phi_d}$ and ${\|Y\|}_{\Phi_d}$ are finite, and applying the Burkholder inequality for Orlicz spaces (see \citealp{MR0566768}, p.304, VII - 92), we get
		\[
		{\|\E_{\bm{1}}[X_{\bm{k}}]\|}_{\Phi_d} \triangleleft \Biggl(\sum
		_{\substack{\bm{u},\bm{v}\geq \bm{k}- \bm{1}\\\bm{u}\neq \bm{v}}}a_{\bm{u},\bm{v}%     
		}^{2}{\|\xi_{\bm{k}-\bm{u}}\|}_{\Phi_d}^2{\|\xi_{\bm{k}-\bm{v}}\|}_{\Phi_d}^2\Biggr)^{\frac{1}{2}}.
		\]
		By stationarity and since ${\|\xi_{\bm{0}}\|}_{\Phi_d} < \infty$, we obtain by using assumption (\ref{cond volt}), that condition \eqref{higher moment phi} holds.
		Thus the CLT in Corollary \ref{Cor_FCLT_d} holds.
	\end{proof}
	Here, we cannot relax assumption \eqref{cond volt} on the coefficients $a_{\bm{u},\bm{v}}$ to a condition similar to \eqref{cond lin} since it would not guarantee that the projective criterion \eqref{hannan} is satisfied. Indeed, in the case of Volterra fields we have
	\[
	\mathcal{P}_{\bm{0}}(X_{\bm{k}}) = \sum_{\substack{\bm{u}, \bm{v} \ge \bm{k}\\ \langle \bm{u} - \bm{k}, \bm{v} - \bm{k} \rangle = 0}}a_{\bm{u}, \bm{v}}\xi_{\bm{k} - \bm{u}}\xi_{\bm{k} - \bm{v}},
	\]
	where $\langle \bm{i},\bm{j} \rangle = \sum_{\ell =1}^d{i_\ell j_\ell}$ for $\bm{i}, \bm{j} \in \Z^d$. Therefore, using the independence of the $\xi_{\bm{n}}$, it holds that
	\[
	{\|\mathcal{P}_{\bm{0}}(X_{\bm{k}})\|}_{\Phi_d} \triangleright {\|\xi_{\bm{0}}\|}_{2}^2\sqrt{\sum_{\substack{\bm{u}, \bm{v} \ge \bm{k}\\ \langle \bm{u} - \bm{k}, \bm{v} - \bm{k} \rangle = 0}}a_{\bm{u}, \bm{v}}^2},
	\]
	Now, if we let $g : x \mapsto x^{-\frac{1}{2}}h(x)$ where $h(x) = \int_1^{x}{\frac{1}{(\log(1+y))^2}dy}$ and $a_{\bm{u}, \bm{v}} = \Bigl(\frac{g'(u_2)g'(v_1)}{u_1v_2}\Bigr)^{1/2}$, then 
	
	\[
	\sum_{\bm{k\geq1}}\frac{1}{\sqrt{|\bm{k}|}}\Biggl(\sum
	_{\bm{u},\bm{v}\geq \bm{k}-\bm{1}}a_{\bm{u},\bm{v}%     
	}^{2}\Biggr)^{1/2} < \infty,
	\]
	but Hannan's condition \eqref{hannan} fails.
	\subsection{Hölderian function of a linear field}
	\par Linear random fields such as the ones presented in Example \ref{linear} are quite useful for stochastic modeling and are a very common occurrence in the literature of that subject. Despite that, these types of fields might not capture the nonlinear properties of many dynamical systems and thus practicians are often required to introduce some more complex models that lack linearity. As we have seen, Volterra fields are a way to do so; however, in a lot of cases, a better model for a dynamical system can appear through random fields which are regular functions of linear fields. In the following example, we are interested in a type of regularity known as H\"{o}lder continuity. Such processes have been studied by \cite{MR2359065} for example and an annealed, i.e. nonquenched, central limit theorem under Hannan's condition has been derived. Here, we will provide sufficient conditions for that central limit theorem to be quenched.
	\begin{example}
		\label{Holder}(Hölder function of a linear field) Consider an Hölder continuous function $f : \mathbb{R} \to \mathbb{R}$ of order $\alpha \in (0,1]$ and let $(\xi_{\bm{n}})_{\bm{n}\in \Z^{d}%  
		}$ be a random field of independent, identically distributed, and centered random variables which satisfy the following condition: 
		\[\begin{cases} \E[\xi_{\bm{0}}^{2}] <\infty\text{ if } \alpha\in (0,1),\\
			\E\Bigl[|\xi_{\bm{0}}|^{2}\bigl(\log (1 + |\xi_{\bm{0}}|)\bigr)^{d-1}\Bigr] <\infty \text{ if } \alpha = 1.
		\end{cases}\] 
		For $\bm{k}\geq\bm{0}$, define%                                         
		\begin{equation}
			\label{Holder def}
			X_{\bm{k}}=f\Biggl(\sum_{\bm{j}\geq\bm{0}}a_{\bm{j}}\xi_{\bm{k}%    
				-\bm{j}}\Biggr) - \mathbb{E}\Biggl[f\Biggl(\sum_{\bm{j}\geq\bm{0}}a_{\bm{j}}\xi_{\bm{k}%    
				-\bm{j}}\Biggr)\Biggr]
		\end{equation}
		where $a_{\bm{u}}$ are real coefficients such that $\sum_{\bm{u}\geq\bm{0}}a_{\bm{u}}^{2}         
		<\infty$. If the coefficients $a_{\bm{u}}$ also satisfy
		\begin{equation}
			\label{cond holder}
			\sum_{\bm{k}\geq\bm{1}}\frac{1}{{\Phi_d}^{-1}(|\bm{k}|)}\Biggl(\sum
			_{\bm{j}\geq\bm{k-1}}a_{\bm{j}}^{2}\Biggr)^{\frac{\alpha}{2}}<\infty
		\end{equation}
		then the quenched functional CLT in Corollary \ref{Cor_FCLT_d} holds.
	\end{example}
	Once again, note that \eqref{cond holder} is satisfied whenever 
	\begin{equation*}
		\sum_{\bm{k}\geq\bm{1}}\frac{\bigl(\log(|\bm{k}|)\bigr)^{\frac{d-1}{2}}}{|\bm{k}|^{\frac{1}{2}}}\Biggl(\sum
		_{\bm{j}\geq\bm{k-1}}a_{\bm{j}}^{2}\Biggr)^{\frac{\alpha}{2}}<\infty.
	\end{equation*}
	Before moving on with the rest of the proof we will need the following lemma whose proof will be provided later on. 
	\begin{lemma}
		\label{lemme_tech}
		For all $\alpha \in (0,1)$ and for any nonnegative random variable $X$, we have the following bound
		\[
		{\|X^\alpha\|}_{\Phi_d} \le K_{\alpha, d} {\|X\|}_2^{\alpha}
		\]
		with $K_{\alpha, d} > 0$ only depending on $\alpha$ and $d$.
	\end{lemma}	
	\begin{proof}[Proof of Example~\ref{Holder}]
		Throughout this proof, we will denote by $C_\alpha$ the Hölder constant associated with $f$. 
		\\
		
		Let $\bm{k} \ge \bm{1}$ and consider $(\xi'_{\bm{n}})_{\bm{n}\in \Z^{d}}$ an independent copy of $(\xi_{\bm{n}})_{\bm{n}\in \Z^{d}}$. Using the fact that $\mathbb{E}_{\bm{1}}[X'_{\bm{k}}] = 0$ where $X'_{\bm{k}}$ is given by \eqref{Holder def} with the innovations $\xi_{\bm{n}}$ replaced by $\xi'_{\bm{n}}$, we deduce that
		\[
		\mathbb{E}_{\bm{1}}[X_{\bm{k}}] = \mathbb{E}_{\bm{1}}\Biggl[f\Biggl(\sum_{\bm{j}\geq\bm{0}}a_{\bm{j}}\xi^*_{\bm{k}%    
			-\bm{j}}\Biggr)\Biggr] - \mathbb{E}_{\bm{1}}\Biggl[f\Biggl(\sum_{\bm{j}\geq\bm{0}}a_{\bm{j}}\xi'_{\bm{k}%    
			-\bm{j}}\Biggr)\Biggr],
		\]
		where $\xi^*_{\bm{n}} = \xi_{\bm{n}}$ for $\bm{n} \le \bm{1}$ and $\xi^*_{\bm{n}} = \xi'_{\bm{n}}$ otherwise. Since $f$ is Hölder continuous of order $\alpha$, we find that
		\begin{align*}
			{\|\mathbb{E}_{\bm{1}}[X_{\bm{k}}]\|}_{\Phi_d} & \le C_{\alpha}{\Biggl\|\Biggl|\sum_{\bm{j}\ge\bm{k} - \bm{1}}a_{\bm{j}}(\xi_{\bm{k}%    
					-\bm{j}}-\xi'_{\bm{k}%    
					-\bm{j}})\Biggr|^{\alpha}\Biggr\|}_{\Phi_d}.
		\end{align*}
		First, suppose that \underline{$0 <\alpha < 1$}. Then by Lemma \ref{lemme_tech}, there exists $K_{\alpha, d} > 0$ such that
		\[
		{\Bigg\|\Biggl|\sum_{\bm{j}\ge\bm{k} - \bm{1}}a_{\bm{j}}(\xi_{\bm{k}%    
				-\bm{j}}-\xi'_{\bm{k}%    
				-\bm{j}})\Biggr|^{\alpha}\Biggr\|}_{\Phi_d} \le K_{\alpha, d} {\Bigg\|\sum_{\bm{j}\ge\bm{k} - \bm{1}}a_{\bm{j}}(\xi_{\bm{k}%    
				-\bm{j}}-\xi'_{\bm{k}%    
				-\bm{j}})\Biggr\|}_{2}^{\alpha}.
		\]
		However, according to the Burkholder inequality for the $L^2$-norm, it holds
		\[
		{\Bigg\|\sum_{\bm{j}\ge\bm{k} - \bm{1}}a_{\bm{j}}(\xi_{\bm{k}%    
				-\bm{j}}-\xi'_{\bm{k}%    
				-\bm{j}})\Biggr\|}_{2} 
		\triangleleft 
		\Bigg(\sum_{\bf{j}\ge\bf{k} - \bf{1}}a_{\bf{j}}^2{\|\xi_{\bf{k}-\bf{j}}\|}_{2}^2\Bigg)^{\frac{1}{2}} ={\|\xi_{\bm{0}}\|}_{2}\Biggl(\sum_{\bf{j}\ge\bf{k} - \bf{1}}a_{\bf{j}}^2\Biggr)^{\frac{1}{2}}
		\]
		Thus
		\begin{equation}
			\label{ineq case <1}
			{\Biggl\|\Biggl|\sum_{\bm{j}\ge\bm{k} - \bm{1}}a_{\bm{j}}(\xi_{\bm{k}%    
					-\bm{j}}-\xi'_{\bm{k}%    
					-\bm{j}})\Biggr|^{\alpha}\Biggr\|}_{\Phi_d} 
			\triangleleft {\|\xi_{\bm{0}}\|}_{2}^{\alpha} \Biggl(\sum_{\bm{j} \ge \bm{k} - \bm{1}}a_{\bm{j}}^2\Biggr)^{\frac{\alpha}{2}}.
		\end{equation}
		Now, suppose that \underline{$\alpha = 1$}, then applying the Burkholder inequality for Orlicz space (see \citealp{MR0566768}, p.304, VII - 92), we get
		\begin{equation}
			\label{ineq case =1}
			{\Bigg\|\sum_{\bm{j}\ge\bm{k} - \bm{1}}a_{\bm{j}}(\xi_{\bm{k}%    
					-\bm{j}}-\xi'_{\bm{k}%    
					-\bm{j}})\Biggr\|}_{\Phi_d} \triangleleft  {\|\xi_{\bm{0}}\|}_{\Phi_d} \Biggl(\sum_{\bm{j} \ge \bm{k} - \bm{1}}a_{\bm{j}}^2\Biggr)^{\frac{1}{2}}.
		\end{equation}
		Combining both inequalities \eqref{ineq case <1} and \eqref{ineq case =1}, we obtain that for any $\alpha \in (0,1]$, the inequality
		\[
		{\Biggl\|\Bigg|\sum_{\bm{j}\ge\bm{k} - \bm{1}}a_{\bm{j}}(\xi_{\bm{k}%    
				-\bm{j}}-\xi'_{\bm{k}%    
				-\bm{j}})\Bigg|^{\alpha}\Biggr\|}_{\Phi_d} 
		\triangleleft 
		\begin{cases}
			{\|\xi_{\bm{0}}\|}_{2}^{\alpha} \Bigl(\sum_{\bm{j} \ge \bm{k} - \bm{1}}a_{\bm{j}}^2\Bigr)^{\frac{\alpha}{2}} \text{ if } \alpha\in (0,1),\\
			{\|\xi_{\bm{0}}\|}_{\Phi_d}^{\alpha} \Bigl(\sum_{\bm{j} \ge \bm{k} - \bm{1}}a_{\bm{j}}^2\Bigr)^{\frac{\alpha}{2}} \text { if } \alpha=1 
		\end{cases}
		\]
		is satisfied. Therefore, using \eqref{cond holder} and the moment condition on the random variable $\xi_{\bm{0}}$, we obtain that \eqref{higher moment phi} is verified and thus the quenched functional CLT in Corollary \ref{Cor_FCLT_d} holds.
	\end{proof}
	In order to prove Lemma \ref{lemme_tech}, we give another very useful property of the natural logarithm: if $0 < \alpha < 1$ and $\epsilon\in (0,\alpha]$, then there exists a constant $c_{d,\epsilon} > 0$ such that for any $x > 0$,
	\begin{equation}
		\bigl(\log(1+x^\alpha)\bigr)^{d-1} \le c_{d,\epsilon} x^{(d-1)\epsilon}.
	\end{equation}
	In particular, this implies that the function $\Phi_d$ satisfies
	\[
	\Phi_d(x^\alpha)\le c_{d,\epsilon}x^{2\alpha+(d-1)\epsilon},
	\]
	for any $x > 0$.
	\begin{proof}[Proof of Lemma~\ref{lemme_tech}]
		Let $\alpha \in (0,1)$ and $X$ be a positive random variable, and consider $\epsilon\in (0,\alpha]$. Consider $t=\kappa^{\alpha}>0$ and remark that
		\[
		\mathbb{E}\biggl[\Phi_d\biggl( \frac{X^\alpha}{t}\biggr) \biggr]=
		\mathbb{E}\biggl[\Phi_d\biggl(\biggl(\frac{X}{\kappa}\biggr)^\alpha\biggr) \biggr]\le c_{d,\epsilon}\mathbb{E}\biggl[\biggl(\frac{X}{\kappa}\biggr)^{2\alpha +(d-1)\epsilon}\biggr].
		\]
		First, suppose that \underline{$\alpha\in [1/2,1 )$}.  We choose $\epsilon$ small enough such that $\delta := 2\alpha+(d-1)\epsilon<2$ (which is always possible since $\alpha<1$). Letting $\kappa_0 =c_{d,\epsilon}^{1/\delta}{\|X\|}_{\delta}$ (correspondingly $t_0=c_{d,\epsilon}^{\alpha/\delta}{\|X\|}_{\delta}^{\alpha}$), we have 
		\[
		\E\biggl[
		\Phi_d\biggl( \frac{X^\alpha}{t_0}\biggr)
		\biggr]
		\le c_{d,\epsilon}\E\biggl[\biggl(\frac{X}{\kappa_0}\biggr)^\delta\biggr] = \frac{\E[X^{\delta}]}{{\|X\|}_{\delta}^{\delta}} = 1.
		\]
		We deduce from the definition of ${\|\cdot\|}_{\Phi_d}$ that 
		\[
		{\|X^{\alpha}\|}_{\Phi_d}
		\le 
		c_{d,\epsilon}^{\alpha/\delta}{\|X\|}_{\delta}^{\alpha}
		\le
		c_{d,\epsilon}^{\alpha/\delta}{\|X\|}_{2}^{\alpha}.
		\]
		Now suppose \underline{$\alpha\in (0,1/2 )$}.  We choose $\epsilon$ small enough such that $\delta := 2\alpha+(d-1)\epsilon<1$ (which is always possible since $\alpha<1/2$). In this case, we set $\kappa_0 =c_{d,\epsilon}^{1/\delta}{\|X\|}_{1}$ (correspondingly $t_0=c_{d,\epsilon}^{\alpha/\delta}{\|X\|}_{1}^{\alpha}$). Then, using the concavity of the function $x \mapsto x^{\delta}$, we get
		\[
		\E\biggl[
		\Phi_d\biggl( \frac{X^\alpha}{t_0}\biggr)
		\biggr]
		\le c_{d,\epsilon}\E\biggl[\biggl(\frac{X}{\kappa_0}\biggr)^\delta\biggr]\le   c_{d,\epsilon}\biggl(\E\biggl[\frac{X}{\kappa_0}\biggr]\biggr)^\delta = \biggl( \frac{\mathbb{E}[X]}{{\|X\|}_{1}}\biggr)^\delta = 1.
		\]
		Therefore, 
		\[
		{\|X^{\alpha}\|}_{\Phi_d}
		\le 
		c_{d,\epsilon}^{\alpha/\delta}{\|X\|}_{1}^{\alpha}
		\le
		c_{d,\epsilon}^{\alpha/\delta}{\|X\|}_{2}^{\alpha}.
		\]
		
		Combining the discussions above, we conclude that for any $\alpha\in(0,1)$ there exists $K_{\alpha,d} > 0$ such that
		
		\[
		{\|X^{\alpha}\|}_{\Phi_d}
		\le
		K_{\alpha,d}{\|X\|}_{2}^{\alpha}.
		\]
	\end{proof}
	\par As in the previous example, we cannot relax condition \eqref{cond holder} to condition \eqref{cond lin} in this case. In fact, neither \eqref{cond lin} nor the condition $\sum_{\bm{i} \ge \bm{0}}|a_{\bm{i}}|^{\alpha} < \infty$ are enough to guarantee that \eqref{hannan} holds. Indeed, consider the case $d = 2$ and let 
	\[a_{u,v} = \biggl\{\begin{array}{l r}
		g'(u)g'(v) & \textrm{if $u > 0$ and $v > 0$}\\
		0 & \textrm{otherwise,}\\
	\end{array}
	\]
	where $g : x \mapsto x^{-1}\bigl(\log(1+x)\bigr)^{-3}$. 
	\\
	\\
	Assume that the innovation field $(\xi_{\bm{i}})_{\bm{i} \in\Z^2}$ is a random field of independent and identically distributed random variables such that $\xi_{0,0}$ follows the standard normal distribution $\mathcal{N}(0,1)$. Consider the Lipschitz (Hölderian of order $1$) function $f : x \in \mathbb{R} \mapsto |x| \in \mathbb{R}^+$. Letting $i,j \ge 0$, we have 
	\[
	\mathcal{P}_{0,0}(X_{i,j}) = \mathbb{E}_{0,0}\bigl[f(Y + Z + \zeta_{0,0}) - f(Y + Z +\zeta_{-1,0}) - f(Y + Z +\zeta_{0,-1}) + f(Y + Z + \zeta_{-1,-1})\bigr]
	\]
	with
	\[
	Y = \sum_{k \ge i+1}\sum_{l \ge j+1}a_{k,l}\xi_{i-k,j-l}, \quad Z = \sum_{k \ge 0}\sum_{l \ge 0}a_{k,l}\xi'_{i-k,j-l} - \sum_{k \ge i}\sum_{l \ge j}a_{k,l}\xi'_{i-k,j-l},
	\]
	\[
	\zeta_{0,0} = \sum_{k \ge i}a_{k,j}\xi_{i-k,0} + \sum_{l \ge j+1}a_{i,l}\xi_{0,j-l},
	\]
	\[
	\zeta_{-1,0} = \sum_{k \ge i+1}a_{k,j}\xi_{i-k,0} + \sum_{l \ge j}a_{i,l}\xi_{0,j-l}',
	\]
	\[
	\zeta_{0,-1} = \sum_{k \ge i}a_{k,j}\xi_{i-k,0}' + \sum_{l \ge j+1}a_{i,l}\xi_{0,j-l},\]
	and
	\[
	\zeta_{-1,-1} = \sum_{k \ge i}a_{k,j}\xi_{i-k,0}' + \sum_{l \ge j+1}a_{i,l}\xi_{0,j-l}'.
	\]
	where $(\xi'_{i,j})_{(i,j) \in \mathbb{Z}^2}$ is  an independent copy of $(\xi _{i,j})_{(i,j) \in \mathbb{Z}^2}$. Let $x \in {\mathbb{R}}^{\mathbb{Z}^2}$, then
	\begin{align*}
		& \mathbb{E}\bigl[f(Y + Z + \zeta_{0,0}) - f(Y + Z +\zeta_{-1,0}) - f(Y + Z +\zeta_{0,-1}) + f(Y + Z + \zeta_{-1,-1})\bigm|\xi = x\bigr]\\
		& \qquad \qquad = \mathbb{E}\bigl[f(y + Z+ \zeta_{0,0}^x) - f(y + Z + \zeta_{-1,0}^x) - f(y + Z + \zeta_{0,-1}^x) + f(y + Z + \zeta_{-1,-1}^x)\bigr]
	\end{align*}
	where $y = \sum_{k \ge i+1}\sum_{l \ge j+1}a_{k,l}x_{i-k,j-l}$, 
	\[
	\zeta_{0,0}^x = \sum_{k \ge i}a_{k,j}x_{i-k,0} + \sum_{l \ge j+1}a_{i,l}x_{0,j-l},
	\]
	\[
	\zeta_{-1,0}^x = \sum_{k \ge i+1}a_{k,j}x_{i-k,0} + \sum_{l \ge j}a_{i,l}\xi'_{0,j-l}
	\]
	\[
	\zeta_{0,-1}^x = \sum_{k \ge i}a_{k,j}\xi'_{i-k,0} + \sum_{l \ge j+1}a_{i,l}x_{0,j-l}
	\]
	and
	\[
	\zeta_{-1,-1}^x = \sum_{k \ge i}a_{k,j}\xi_{i-k,0}' + \sum_{l \ge j+1}a_{i,l}\xi_{0,j-l}'.
	\]
	Remark that $y + Z+ \zeta_{0,0}^x$, $y + Z + \zeta_{-1,0}^x$, $y + Z + \zeta_{0,-1}^x$ and $y + Z + \zeta_{-1,-1}^x$ follows respectively the normal distributions  $\mathcal{N}\bigl(m_{0,0}^x, \sigma_{0,0}^2\bigr)$, $\mathcal{N}\bigl(m_{-1,0}^x, \sigma_{-1,0}^2\bigr)$, $\mathcal{N}\bigl(m_{0,-1}^x, \sigma_{0,-1}^2\bigr)$ and $\mathcal{N}\bigl(m_{-1,-1}^x, \sigma_{-1,-1}^2\bigr)$ with
	\[
	\sigma_{0,0}^2 = \Var[Z] = \sum_{k \ge 0}\sum_{l \ge 0}a_{k,l}^2 - \sum_{k \ge i}\sum_{l \ge j}a_{k,l}^2 \quad \textrm{ and } \quad m_{0,0}^x = y + \zeta_{0,0}^x,
	\]
	\[
	\sigma_{-1,0}^2 = \Var[Z + \zeta_{-1,0}^x] = \sum_{k \ge 0}\sum_{l \ge 0}a_{k,l}^2 - \sum_{k \ge i+1}\sum_{l \ge j}a_{k,l}^2\quad \textrm{ and } \quad m_{-1,0}^x = y + \sum_{k \ge i+1}a_{k,j}x_{i-k,0},
	\]
	\[
	\sigma_{0,-1}^2 = \Var[Z + \zeta_{0,-1}^x] = \sum_{k \ge 0}\sum_{l \ge 0}a_{k,l}^2 - \sum_{k \ge i}\sum_{l \ge j+1}a_{k,l}^2\quad \textrm{ and } \quad m_{0,-1}^x = y + \sum_{l \ge j+1}a_{i,l}x_{0,j-l},
	\]
	and
	\[
	\sigma_{-1,-1}^2 = \Var[Z + \zeta_{-1,-1}^x] = \sum_{k \ge 0}\sum_{l \ge 0}a_{k,l}^2 - \sum_{k \ge i+1}\sum_{l \ge j+1}a_{k,l}^2\quad \textrm{ and } \quad m_{-1,-1}^x = y.
	\]
	Thus, for $(a,b) \in \{-1,0\}^2$,
	\[
	\mathbb{E}[f(y + Z+ \zeta_{a,b}^x)] = \mathbb{E}[|y + Z+ \zeta_{a,b}^x|] = \frac{2\sigma_{a,b}}{\sqrt{2\pi}}e^{-\frac{(m_{a,b}^x)^2}{2\sigma_{a,b}^2}} + \frac{|m_{a,b}^x|}{\sqrt{2\pi \sigma_{a,b}^2}}\int_{-|m_{a,b}^x|}^{|m_{a,b}^x|}e^{-\frac{z^2}{2\sigma_{a,b}^2}}dz.
	\]
	Rewriting the right-hand side, we notice that
	\begin{align*}
		& \frac{2\sigma_{a,b}}{\sqrt{2\pi}}e^{-\frac{(m_{a,b}^x)^2}{2\sigma_{a,b}^2}} + \frac{|m_{a,b}^x|}{\sqrt{2\pi \sigma_{a,b}^2}}\int_{-|m_{a,b}^x|}^{|m_{a,b}^x|}e^{-\frac{z^2}{2\sigma_{a,b}^2}}dz\\
		& \qquad \qquad  = \sigma_{a,b}\Biggl(\frac{2}{\sqrt{2\pi}}e^{-\frac{(m_{a,b}^x)^2}{2\sigma_{a,b}^2}} + \frac{|m_{a,b}^x|}{\sigma_{a,b}}\int_{-\frac{|m_{a,b}^x|}{\sigma_{a,b}}}^{\frac{|m_{a,b}^x|}{\sigma_{a,b}}}\frac{e^{-\frac{z^2}{2}}}{\sqrt{2\pi}}dz\Biggr)\\
		& \qquad \qquad =: \sigma_{a,b}g\biggl(\frac{m_{a,b}^x}{\sigma_{a,b}}\biggr).
	\end{align*}
	For any $(a,b) \in \{-1,0\}^2$, we let $\Xi_{a,b} = m_{a,b}^{\xi} - Y$ and by applying Taylor's formula, we get
	\begin{align*}
		g\biggl(\frac{m_{a,b}^\xi}{\sigma_{a,b}}\biggr) & = g\biggl(\frac{Y}{\sigma_{0,0}} + \bigg(\frac{1}{\sigma_{a,b}} - \frac{1}{\sigma_{0,0}}\biggr)Y + \frac{\Xi_{a,b}}{\sigma_{a,b}}\biggr)\\
		& = g\biggl(\frac{Y}{\sigma_{0,0}}\biggr) + \frac{\Xi_{a,b}}{\sigma_{a,b}}g'\biggl(\frac{Y}{\sigma_{0,0}}\biggr) + \biggl(\frac{1}{\sigma_{a,b}} - \frac{1}{\sigma_{0,0}}\biggr)Yg'\biggl(\frac{Y}{\sigma_{0,0}}\biggr)  + R_{a,b}.
	\end{align*}
	where $|R| \le \frac{2\Xi_{a,b}^2}{\sigma_{a,b}^2} + 2\Bigl(\frac{1}{\sigma_{a,b}} - \frac{1}{\sigma_{0,0}}\Bigr)^2Y^2$. Thus
	\begin{align*}
		{\|\mathcal{P}_{0,0}(X_{i,j})\|}_{\Phi_d} & = {\Biggl\|\sum_{(a,b) \in \{0,1\}^2}(-1)^{a+b}\sigma _{a,b}g\biggl(\frac{m_{a,b}^\xi}{\sigma_{a,b}}\biggr)\Biggr\|}_{\Phi_d}\\
		& \ge \underset{=:N_1}{\underbrace{{\Bigg\|g\biggl(\frac{Y}{\sigma_{0,0}}\biggr)\sum_{(a,b) \in \{0,1\}^2}(-1)^{a+b}\sigma _{a,b}\Biggr\|}_{\Phi_d}}} -  \underset{=:N_2}{\underbrace{{\Biggl\|g'\biggl(\frac{Y}{\sigma_{0,0}}\biggr)\sum_{(a,b) \in \{0,1\}^2}(-1)^{a+b}\Xi_{a,b}\Biggr\|}_{\Phi_d}}}\\
		& \qquad - \underset{=:N_3}{\underbrace{{\Biggl\|Yg'\biggl(\frac{Y}{\sigma_{0,0}}\biggr)\sum_{(a,b) \in \{0,1\}^2}(-1)^{a+b}\biggl(1 - \frac{\sigma_{a,b}}{\sigma_{0,0}}\biggr)\Biggr\|}_{\Phi_d}}}\hspace{-0.3cm} - \underset{=:N_4}{\underbrace{{\Biggl\|\sum_{(a,b) \in \{0,1\}^2}(-1)^{a+b}\sigma _{a,b}R_{a,b}\Biggr\|}_{\Phi_d}}}.
	\end{align*}
	We have the following relations
	\[
	N_1 = {\biggl\|g\biggl(\frac{Y}{\sigma_{0,0}}\biggr)\biggr\|}_{\Phi_d}|\sigma_{0,0} - \sigma_{-1,0}- \sigma_{0,-1} + \sigma_{-1,-1}|,
	\]
	\[
	N_2 = {\biggl\|g'\biggl(\frac{Y}{\sigma_{0,0}}\biggr)\xi_{0,0}'\biggr\|}_{\Phi_d}|a_{i,j}|,
	\]
	\[
	N_3 = {\biggl\|\frac{Y}{\sigma_{0,0}}g'\biggl(\frac{Y}{\sigma_{0,0}}\biggr)\biggr\|}_{\Phi_d}|\sigma_{0,0} - \sigma_{-1,0}- \sigma_{0,-1} + \sigma_{-1,-1}|.
	\]
	and
	\begin{align*}
		N_4 &\le {\biggl\|\biggl(\frac{Y}{\sigma_{0,0}}\biggr)^2\biggr\|}_{\Phi_d}\biggl|\frac{(\sigma_{0,0} - \sigma_{-1,0})^2}{\sigma_{-1,0}} + \frac{(\sigma_{0,0} - \sigma_{0,-1})^2}{\sigma_{0,-1}} - \frac{(\sigma_{0,0} - \sigma_{-1,-1})^2}{\sigma_{-1,-1}}\biggr|\\
		& \qquad \qquad + 2\Biggl(\frac{{\|\Xi_{0,0}^2\|}_{\Phi_d}}{\sigma_{0,0}} + \frac{{\|\Xi_{-1,0}^2\|}_{\Phi_d}}{\sigma_{-1,0}} + \frac{{\|\Xi_{0,-1}^2\|}_{\Phi_d}}{\sigma_{0,-1}}\Biggr).
	\end{align*}
	Before continuing the proof, note that the random variable $Y/\sigma_{0,0}$ follows a centered normal distribution with variance inferior to $1$ whenever $i$ and $j$ are large enough. Using the previous relations, it holds that
	\[
	{\|\mathcal{P}_{0,0}(X_{i,j})\|}_{\Phi_d} \ge |\sigma_{0,0} - \sigma_{-1,0}- \sigma_{0,-1} + \sigma_{-1,-1}|\Biggl({\biggl\|g\biggl(\frac{Y}{\sigma_{0,0}}\biggr)\biggr\|}_{\Phi_d} - {\biggl\|\frac{Y}{\sigma_{0,0}}g'\biggl(\frac{Y}{\sigma_{0,0}}\biggr)\biggr\|}_{\Phi_d}\Biggr) - N_2 - N_4.
	\]
	Now, making use of the equality $u_1 - u_2 = (u_1^2 - u_2^2)/(u_1 + u_2)$ for any $u_1,u_2 > 0$, we obtain
	\[
	\sigma_{0,0} - \sigma_{-1,0} - \sigma_{0,-1} + \sigma_{-1,-1} = \frac{\sigma_{0,0}^2 - \sigma_{-1,0}^2}{\sigma_{0,0} + \sigma_{-1,0}} - \frac{\sigma_{0,-1}^2 - \sigma_{-1,-1}^2}{\sigma_{0,-1} + \sigma_{-1,-1}}.
	\]
	However,
	\[
	\delta_{i,j} := |\sigma_{0,0}^2 - \sigma_{-1,0}^2| = \sum_{l \ge j}a_{i,l}^2 \quad \textrm{ and } \quad |\sigma_{0,-1}^2 - \sigma_{-1,-1}^2| = \delta_{i,j} - a_{i,j}^2.
	\]
	According to the definition of the coefficients $a_{i,j}$, it is possible to show that 
	\[
	|\sigma_{0,0} - \sigma_{-1,0} - \sigma_{0,-1} + \sigma_{-1,-1}| \ge \biggl(\frac{1}{\sigma_{0,0} + \sigma_{-1,0}} - \frac{1}{\sigma_{0,-1} + \sigma_{-1,-1}}\biggr)\delta_{i,j} \triangleright \bigl(\log(1+i)\log(1+j)\bigr)^{-3/2}\sqrt{j}.
	\]
	In particular, it holds that
	\[
	\sum_{(i,j)\ge (0,0)}{\|\mathcal{P}_{0,0}(X_{i,j})\|}_{\Phi_d} = \infty,
	\]
	while also having both
	\[
	\sum_{(i,j) \ge (0,0)}|a_{i,j}| < \infty \quad \textrm{and} \quad \sum_{(i,j)\geq(1,1)}\frac{1}{\sqrt{ij}}\Biggl(\sum
	_{(u,v)\geq(i-1,j-1)}a_{i,j}^{2}\Biggr)^{\frac{1}{2}}<\infty.
	\]
	\subsection{Weakly Dependent Processes}
	In our last example, we study the quenched central limit theorem for weakly dependent random fields in the sense of Wu. Fields of this kind were introduced by \cite{Wu2005} and have many applications in mathematical physics, especially within the study of particle systems (see \citealp{MR2108619, MR1182416}). Indeed, weakly dependent random fields are particularly well-suited to model physical systems as they can capture, at least partially, the influence of the inputs over the outputs of these systems. In particular, they are well adapted to study the case of nonlinear physical models.
	\\
	\par Consider a centered Bernoulli random field $(X_{\bm{i}})_{\bm{i}\in\Z^d}$ defined for every $\bm{i} \in \Z^d$ by $X_{\bm{i}}:= G(\xi_{\bm{i} - \bm{s}};{\bm{s} \ge \bm{0}})$ where $(\xi_{\bm{i}})_{\bm{i}\in\Z^d}$ is a field of independent and identically distributed random variables. Now, denote by $(\xi_{\bm{i}}')_{\bm{i}\in\Z^d}$ an independent copy of $(\xi_{\bm{i}})_{\bm{i}\in\Z^d}$ and set, for any $\bm{i} \in \Z^d$, 
	\[
	\biggl\{\begin{array}{l r}
		\xi_{\bm{i}}^* = \xi_{\bm{i}}' & \textrm{if $\bm{i} = \bm{0}$}\\
		\xi_{\bm{i}}^* = \xi_{\bm{i}} & \textrm{otherwise}
	\end{array} \quad \textrm{as well as} \quad \biggl\{\begin{array}{l r}
		\widetilde{\xi}_{\bm{i}} = \xi_{\bm{i}}' & \textrm{if $i_1 = 0$ and $\bm{i} \le \bm{1}$}\\
		\widetilde{\xi}_{\bm{i}} = \xi_{\bm{i}} & \textrm{otherwise.}
	\end{array}
	\]
	Then, the perturbed systems $(X_{\bm{i}}^*)_{\bm{i} \in \Z^d}$ and $(\widetilde{X}_{\bm{i}})_{\bm{i} \in \Z^d}$ are given by 
	\[
	X_{\bm{i}}^* = G(\xi_{\bm{i} - \bm{s}}^*;{\bm{s} \ge \bm{0}}) \quad \textrm{and} \quad \widetilde{X}_{\bm{i}} = G(\widetilde{\xi}_{\bm{i} - \bm{s}};{\bm{s} \ge \bm{0}}), \qquad \bm{i} \in \Z^d.
	\]
	In this subsection, we are interested in two different stability conditions. First, we take a look at the usual notion of weak dependence in the sense of Wu by saying that the random field $(X_{\bm{i}})_{\bm{i}\in\Z^d}$ is stable whenever 
	\[
	\sum_{\bm{i} \ge \bm{0}}\delta_{\bm{i}} < \infty
	\]
	where the terms $\delta_{\bm{i}}$ are known as the physical dependence coefficients and are defined by
	\[
	\delta_{\bm{i}} = {\|X_{\bm{i}} - X_{\bm{i}}^*\|}_{\Phi_d}.
	\] 
	Under this notion of weak dependence, we have the following quenched functional central limit theorem.
	\begin{example}
		\label{stable}
		Suppose that the field $(X_{\bm{i}})_{\bm{i}\in\Z^d}$ satisfies 
		\begin{equation}
			\label{cond regular}
			\mathbb{E}[X_{\bm{0}}|\mathcal{G}_n]\converge{n}{\infty}{\normalfont{a.s.}} 0,
		\end{equation}
		where $\mathcal{G}_n = \sigma\bigl(\xi_{\bm{i}} : \bm{i} \le \bm{1} \textrm{ and } \exists k \in \llbracket 1,d\rrbracket, i_k \le n \bigr)$. Additionally, suppose that
		\begin{equation}
			\label{cond weak 1}
			\sum_{\bm{k} \ge \bm{1}}\frac{1}{\Phi_d^{-1}(|\bm{k}|)}\sum_{\bm{j} \ge \bm{k} - \bm{1}}\delta_{\bm{j}} < \infty.
		\end{equation}
		Then the conclusion of Corollary~\ref{Cor_FCLT_d} holds.
	\end{example}
	Note that the condition \eqref{cond regular} is a stronger condition than the regularity of $X_{\bm{0}}$. Moreover, as we have seen in Example \ref{Volterra} and Example \ref{Holder}, it is possible to give a stronger yet more tractable condition than \eqref{cond weak 1} which is stated below. Indeed, if  \[\sum_{\bm{k}\geq\bm{1}}\frac{\bigl(\log(|\bm{k}|)\bigr)^{\frac{d-1}{2}}}{|\bm{k}|^{\frac{1}{2}}}\sum
	_{\bm{j}\geq\bm{k-1}}\delta_{\bm{j}}<\infty\]
	then the conclusion of Example \ref{stable} is verified.
	\\
	\begin{proof}
		Let $\bm{k} \ge \bm{1}$ and consider a bijection $\tau : \mathbb{Z} \to \Z^d$ such that for all $n \in \Z$, we have 
		\[\bigl(n \ge 1\; \Longleftrightarrow\; \tau(n) \in \Z^-_{\bm{1}}\bigr) \quad \textrm{and} \quad \bigl(
		\forall k \in \mathbb{N}, \{\bm{i} \in \Z^d : 1 \le \tau^{-1}(\bm{i}) \le k^d\} = \llbracket2-k,1\rrbracket^d\bigr)
		\]
		where $\Z^-_{\bm{1}} = \{\bm{i} \in \Z^d : \bm{i} \le \bm{1}\}$. Since $X_{\bm{0}}$ is centered and satisfies \eqref{cond regular}, we find that 
		\begin{equation}
			\E_{\bm{1}}[X_{\bm{k}}] = \E_{\bm{1}}[X_{\bm{k}}] - \E_{\bm{1}}\bigl[G\bigl(\xi'_{\bm{k} - \bm{s}};{\bm{s} \ge \bm{0}}\bigr)\bigr] =  \sum_{n \ge 0}\E_{\bm{1}}[Y_{\tau(n)} - Y_{\tau(n+1)}],
		\end{equation}
		where 
		\[Y_{\tau(n)} = G\bigl(\zeta^{n}_{\bm{k} - \bm{s}};{\bm{s} \ge \bm{0}}\bigr), \quad \textrm{with}\quad \zeta^{n}_{\bm{i}} = \biggl\{\begin{array}{l r}
			\xi_{\bm{i}} & \textrm{if $\tau^{-1}(\bm{i}) > n$}\\
			\xi'_{\bm{i}} & \textrm{otherwise}.
		\end{array}
		\]
		Remark that, according to stationarity, for any $n \ge 0$
		\begin{align*}
			{\|\E_{\bm{1}}[Y_{\tau(n)} - Y_{\tau(n+1)}]\|}_{\Phi_d} & \le {\bigl\|G\bigl(\zeta^{n}_{\bm{k} - \bm{s}};{\bm{s} \ge \bm{0}}\bigr) - G\bigl(\zeta^{n+1}_{\bm{k} - \bm{s}};{\bm{s} \ge \bm{0}}\bigr)\bigr\|}_{\Phi_d}\\
			& = {\bigl\|G\bigl(\zeta^{n}_{\bm{k} - \bm{s}};{\bm{s} \ge \bm{0}}\bigr) - G\bigl(\zeta^{n}_{\bm{k} - \bm{s}}, \xi'_{\tau(n+1)};{\bm{s} \ge \bm{0}, \bm{s} \not = \bm{k} - \tau(n+1)} \bigr)\bigr\|}_{\Phi_d}\\
			& = {\bigl\|G\bigl(\xi_{\bm{k}- \tau(n+1) - \bm{s}};{\bm{s} \ge \bm{0}}\bigr) - G\bigl(\xi_{\bm{k}- \tau(n+1) - \bm{s}}, \xi'_{\bm{0}};{\bm{s} \ge \bm{0}, \bm{s} \not = \bm{k} - \tau(n+1)} \bigr)\bigr\|}_{\Phi_d}\\
			& = \delta_{\bm{k}- \tau(n+1)}.
		\end{align*}
		Therefore, using the triangular inequality, we get
		\[
		{\|\E_{\bm{1}}[X_{\bm{k}}]\|}_{\Phi_d} \le \sum_{n \ge 0}\delta_{\bm{k} - \tau(n+1)} = \sum_{\bm{j}\ge \bm{k} - \bm{1}}\delta_{\bm{j}}.
		\]
		Using condition \eqref{cond weak 1} we find that the conclusion of Corollary \ref{Cor_FCLT_d} is satisfied.
	\end{proof}
	By considering a stronger notion of stability, we can relax the hypothesis \eqref{cond weak 1} to \eqref{cond weak 2} as well as condition \eqref{cond regular} to the simple regularity of $X_{\bm{0}}$. In fact, we will say that a random field $(X_{\bm{i}})_{\bm{i}\in\Z^d}$ is strongly stable whenever 
	\[
	\sum_{\bm{i} \ge \bm{0}}\widetilde{\delta}_{\bm{i}} < \infty
	\]
	with
	\[
	\widetilde{\delta_{\bm{i}}} = {\|X_{\bm{i}} - \widetilde{X}_{\bm{i}}\|}_{\Phi_d}.
	\]
	Under this stronger assumption, we can show that Hannan's condition \eqref{hannan} holds. Then there only remains to satisfy \eqref{higher moment 2} for Corollary \ref{main result no centering} to apply.	
	\begin{example}
		\label{strong stable}
		Suppose that the field $(X_{\bm{i}})_{\bm{i}\in\Z^d}$ is strongly stable and that $X_{\bm{0}}$ is regular, then condition \eqref{hannan} is also satisfied. If, in addition, we also assume that
		\begin{equation}
			\label{cond weak 2}
			\sum_{\bm{k} \ge \bm{1}}\frac{1}{\sqrt{|\bm{k}|}}\sum_{j \ge 0}\widetilde{\delta}_{\bm{k}+(j-1)\bm{e}_1} < \infty,
		\end{equation}
		then the conclusion of Corollary \ref{Cor_FCLT_d} holds.
	\end{example}
	\begin{proof}
		Let $\bm{k} \ge \bm{0}$, then we have the following bound
		\begin{align*}
			{\|\mathcal{P}_{\bm{0}}(X_{\bm{k}})\|}_{\Phi_d} & = {\Biggl\|\prod_{i=1}^{d}{(\E_{\bm{0}} - \E_{-\bm{e}_i})[X_{\bm{k}}]}\Biggr\|}_{\Phi_d}\\
			& \le 2^{d-1}{\|\E_{\bm{0}}[X_{\bm{k}}] - \E_{-\bm{e}_1}[X_{\bm{k}}]\|}_{\Phi_d}\\
			& \le 2^{d-1}\widetilde{\delta}_{\bm{k}}.
		\end{align*}
		Hence \eqref{hannan} is satisfied. Now suppose that \eqref{cond weak 2} holds and $\bm{k} \ge \bm{1}$. Since $G(\xi'_{\bm{k} - \bm{s}}; \bm{s} \ge \bm{0})$ is a centered random variable, we have
		\begin{equation}
			\E_{\bm{1}}[X_{\bm{k}}] = \E_{\bm{1}}[X_{\bm{k}}] - \E_{\bm{1}}\bigl[G\bigl(\xi'_{\bm{k} - \bm{s}};{\bm{s} \ge \bm{0}}\bigr)\bigr] =  \sum_{j \ge 0}\E_{\bm{1}}[Y_j- Y_{j+1}],
		\end{equation}
		where
		\[
		Y_j = G\bigl(\zeta^{j}_{\bm{k} - \bm{s}};{\bm{s} \ge \bm{0}}\bigr), \quad \textrm{with}\quad \zeta^j_{\bm{i}} = \biggl\{\begin{array}{l r}
			\xi'_{\bm{i}} & \textrm{if $1-j < i_1 \le 1$ and $\bm{i \le \bm{1}}$}\\
			\xi_{\bm{i}} & \textrm{otherwise}.
		\end{array}
		\]
		However, using a similar argument as before, we have that 
		\[
		{\|\E_{\bm{1}}[Y_j- Y_{j+1}]\|}_2 \le \widetilde{\delta}_{\bm{k}+(j-1)\bm{e}_1}.
		\]
		Therefore, using the triangle inequality, we get
		\[
		{\|\E_{\bm{1}}[X_{\bm{k}}]\|}_2 \le \sum_{j \ge 0}\widetilde{\delta}_{\bm{k}+(j-1)\bm{e}_1}.
		\]
		And so, using condition \eqref{cond weak 2}, the conclusion of Corollary \ref{main result no centering} holds for the stochastic process $(X_{\bm{k}})_{\bm{k} \in \Z^d}$.
	\end{proof}
	\section{Appendix}
	\label{sec:appendix}
	In this section, we give the proof of Lemma \ref{rosenthal_ineq_orlicz}. We will follow the outline of the proof given by \cite{MR0365692} for the Rosenthal inequality in Lebesgue spaces but first, we need to establish a preliminary lemma concerning the Orlicz norm studied in this document. We start by recalling the definition of the different tools we will require. 
	\\
	\\Recall that the Luxemburg norm associated with the Young function $\Phi_d : x \in [0,\infty) \mapsto \Phi_d(x)=x^{2}\bigl(\log(1+x)\bigr)^{d-1} \in [0,\infty)$ is defined as
	\[
	{\|f\|}_{\Phi} = \inf \Bigl\{t > 0 : \mathbb{E}\bigl[\Phi(|f|/t)\bigr] \le 1\Bigr\},
	\]
	and by $\Psi_d$ we denote the conjugate function associated with $\Phi_d$ defined in the following way
	\[
	\Psi_d(x)=\sup_{y\geq 0}\bigl(xy-\Phi_d(y)\bigr).
	\]
	Besides properties \eqref{prop1 varphi} and \eqref{prop2 varphi}, the natural logarithm also satisfies
	\begin{equation}
		\label{min_phi}
		\log\Bigl(1+\frac{x}{\lambda}\Bigr)\log(1+\lambda)\ge \log(2)\lambda\log(1+x),
	\end{equation}
	for all $0 <\lambda \le 1$ and $x \ge 0$ as well as
	\begin{equation}
		\label{min_phi_2}
		\log\Bigl(1+\frac{x}{\lambda}\Bigr)\log(1+\lambda)\ge \log(2)\frac{\log(1+x)}{\lambda},
	\end{equation} for all $\lambda \ge 1$ and $x \ge 0$. 
	The following lemma will help us compute the Orlicz norm associated with $\Psi_d$ of a specific random variable which will appear in the proof of Lemma \ref{rosenthal_ineq_orlicz}.
	\begin{lemma}
		\label{ineq_norm_spec}
		Suppose that $h \in L^2\log L$ takes nonnegative values. If ${\|h\|}_{\Phi_d} \le 1$, then
		\[
		{\Bigl\|h\bigl(\log(1+h)\bigr)^{d-1}\Bigr\|}_{\Psi_d} \le {\|h\|}_{\Phi_d}.
		\]
		If ${\|h\|}_{\Phi_d} > 1$ then for all $\epsilon \in (0,1)$, there exist a positive constant $C_{d,\epsilon}$ depending only on $d$ and $\epsilon$ such that
		\begin{equation}
			\label{ineq_norm_spec_2}
			{\Bigl\|h\bigl(\log(1+h)\bigr)^{d-1}\Bigr\|}_{\Psi_d} \le C_{d,\epsilon}{\|h\|}_{\Phi_d}^{1+\epsilon}.
		\end{equation}
	\end{lemma}
	Before moving on with the proof of Lemma \ref{ineq_norm_spec}, we explicit another useful property of the natural logarithm. For all $x \ge \lambda \ge 1$,
	\begin{equation}
		\label{ineq_log_2}
		\log\Bigl(1+\frac{x}{\lambda}\Bigr)\log(1+\lambda)\ge \log(2)\log(1+x).
	\end{equation}
	\begin{proof}[Proof of Lemma~\ref{ineq_norm_spec}]
		Let $h \in L^2\log L$ be a nonnegative function such that ${\|h\|}_{\Phi_d} \le 1$ and let $t \in (0,1]$. Using the inequality $\Psi\bigl(x\bigl(\log(1+x)\bigr)^{d-1}/t\bigr) \le \Phi(x/t)$ for all $x \ge 0$, we get
		\[
		\mathbb{E}\biggl[\Psi_d\biggl(\frac{h\log^{d-1}(1+h)}{t}\biggr)\biggr] \le \mathbb{E}\biggl[\Phi_d\biggl(\frac{h}{t}\biggr)\biggr].
		\]
		Taking $t = {\|h\|}_{\Phi_d} \le 1$, we obtain
		\[
		{\Bigl\|h\bigl(\log(1+h)\bigr)^{d-1}\Bigr\|}_{\Psi_d} \le {\|h\|}_{\Phi_d}.
		\]
		We now turn to the proof of the second part of Lemma \ref{ineq_norm_spec} and we begin by noticing that if $ {\|h\|}_{\Phi_d} = \infty$, then \eqref{ineq_norm_spec_2} is trivially satisfied. From now on, we will therefore assume that $1 < {\|h\|}_{\Phi_d} < \infty$. We start by recalling that if we let $\epsilon \in (0,1)$, then there exists a constant $c_{d,\epsilon} > 0$ such that
		\begin{equation}
			\label{ineq_log}
			\Bigl(\log(1+{\|h\|}_{\Phi_d})\Bigr)^{d-1} \le c_{d,\epsilon} {\|h\|}_{\Phi_d}^{\epsilon}.
		\end{equation}
		Now, according to the triangle inequality, we have
		\begin{equation*}
			{\Bigl\|h\bigl(\log(1+h)\bigr)^{d-1}\Bigr\|}_{\Psi_d} \le {\Bigl\|h\bigl(\log(1+h)\bigr)^{d-1} \textbf{1}_{h \le {\|h\|}_{\Phi_d}}\Bigr\|}_{\Psi_d}  + {\Bigl\|h\bigl(\log(1+h)\bigr)^{d-1} \textbf{1}_{h >{\|h\|}_{\Phi_d}}\Bigr\|}_{\Psi_d}.
		\end{equation*}
		To bound the first term in the right-hand side of this inequality, we make use of \eqref{ineq_log} and we obtain
		\begin{align}
			\label{ineq_term_1}
			{\Bigl\|h\bigl(\log(1+h)\bigr)^{d-1} \textbf{1}_{h \le {\|h\|}_{\Phi_d}}\Bigr\|}_{\Psi_d} & \le \Bigl(\log(1+{\|h\|}_{\Phi_d})\Bigr)^{d-1} {\|h\|}_{\Phi_d}\nonumber\\
			& \le c_{d,\epsilon} {\|h\|}_{\Phi_d}^{1+\epsilon}.
		\end{align}
		Dealing with the second term, we combine \eqref{ineq_log_2}, \eqref{ineq_log} and the inequality $\Psi_d(\varphi_d(x)) \le \Phi_d(x)$ where $\varphi_d(x) = x\bigl(\log(1+x)\bigr)^{d-1}$ for all $x \ge 0$, in order to get
		\begin{align*}
			\mathbb{E}\Biggl[\Psi_d\Biggl(\frac{h\bigl(\log(1+h)\bigr)^{d-1}}{\log(2)^{1-d}c_{d,\epsilon}{\|h\|}_{\Phi_d}^{1+\epsilon}}\Biggr) \textbf{1}_{h > {\|h\|}_{\Phi_d}}\Biggr] & \le \mathbb{E}\Biggl[\Psi_d\Biggl(\frac{h\Bigl(\log(2)^{-1}\log\Bigl(1+\frac{h}{{\|h\|}_{\Phi_d}}\Bigr)\log(1+{\|h\|}_{\Phi_d})\Bigr)^{d-1}}{\log(2)^{1-d}c_{d,\epsilon}{\|h\|}_{\Phi_d}^{1+\epsilon}}\Biggr) \Biggr]\\
			& \le \mathbb{E}\Biggl[\Phi_d\Biggl(\frac{h}{{\|h\|}_{\Phi_d}}\Biggr) \Biggr]\\
			& = 1.
		\end{align*}
		Thus
		\begin{equation}
			\label{ineq_term_2}
			{\Bigl\|h\bigl(\log(1+h)\bigr)^{d-1} \textbf{1}_{h >{\|h\|}_{\Phi_d}}\Bigr\|}_{\Psi_d} \le \log(2)^{1-d}c_{d,\epsilon}{\|h\|}_{\Phi_d}^{1+\epsilon}.
		\end{equation}
		Therefore, combining \eqref{ineq_term_1} and \eqref{ineq_term_2} we get the desired result.
	\end{proof}
	\par We can now prove Lemma \ref{rosenthal_ineq_orlicz}. In order to do so, we will make use of Lemma 3.1 in \cite{MR0365692}.
	\\
	\\
	\begin{proof}[Proof of Lemma \ref{rosenthal_ineq_orlicz}]
		We start by introducing a few items of notation.	For all $\bm{n} \in \mathbb{N}^d$, we denote $M_{\bm{n}} = \sum_{\bm{u} = \bm{0}}^{\bm{n} - \bm{1}}d_{\bm{u}}$ and $\sigma_{\bm{n}} = \sqrt{\sum_{\bm{u} = \bm{0}}^{\bm{n} - \bm{1}}d_{\bm{u}}^2}$. Our proof will be split into two parts. In the first part, we will make the additional assumption that the ortho-martingale  $(M_{\bm{n}})_{\bm{n}}$ is nonnegative. Then, in the second part, we will establish the result for real-valued ortho-martingales.
		\\
		\\
		\underline{First step:} 
		We suppose that $(M_{\bm{n}})_{\bm{n} \in (\mathbb{N}^*)^d}$ is a nonnegative ortho-martingale. Let $\bm{n} \in \mathbb{N}^d$ be fixed and remark that, since $\Phi_d(\sqrt{a+b}) \ge \Phi_d(\sqrt{a}) + \Phi_d(\sqrt{b})$ for all $a,b \ge 0$, it holds that
		\begin{equation}
			\label{control_sum_square}
			\mathbb{E}\biggl[\Phi_d\biggl(\frac{\sigma_{\bm{n}}}{\eta}\biggr)\biggr] \ge \sum_{\bm{u} = \bm{0}}^{\bm{n} - \bm{1}}\mathbb{E}\biggl[\Phi_d\biggl(\frac{d_{\bm{u}}}{\eta}\biggr)\biggr].
		\end{equation}
		for any $\eta > 0$.	Let $X = \max(\sigma_{\bm{n}}, \max _{\bm{0} \le \bm{u} \le \bm{n}}M_{\bm{u}})$ and suppose that ${\|X\|}_{\Phi_d} \le 1$. Applying \eqref{control_sum_square} with $\eta = 1$, we get
		\begin{equation}
			\label{def_eta}
			{\|\sigma_{\bm{n}}\|}_{\Phi_d} \le {\|X\|}_{\Phi_d} \le 1 \quad \textrm{and} \quad \eta _0 := \sum_{\bm{u} = \bm{0}}^{\bm{n} - \bm{1}}\mathbb{E}[\Phi_d(d_{\bm{u}})] \le 1.
		\end{equation}
		Setting $\eta_0' = \eta_0\log(2)^{d-1}$, we find that
		\begin{align*}
			\mathbb{E}\biggl[\Phi_d\biggl(\frac{\sigma_{\bm{n}}}{\eta_0'}\biggr)\biggr] & = \frac{1}{{\eta_0'}^2}\mathbb{E}\Biggl[\sigma_{\bm{n}}^2\biggl(\log\biggl(1 + \frac{\sigma_{\bm{n}}}{\eta_0'}\biggr)\biggr)^{d-1}\Biggr]\\
			& \ge \frac{1}{\eta_0}\mathbb{E}[\Phi_d(\sigma_{\bm{n}})]\\
			& \ge 1.
		\end{align*}
		The second to last inequality holds since according to \eqref{min_phi}, we have
		\begin{align*}
			\biggl(\log\biggl(1 + \frac{\sigma_{\bm{n}}}{\eta_0'}\biggr)\biggr)^{d-1} & \ge {\eta_0'}^{d-1}\log(2)^{d-1}\frac{\bigl(\log(1 + \sigma_{\bm{n}})\bigr)^{d-1}}{\bigl(\log(1 + \eta_0')\bigr)^{d-1}}\\
			& \ge {\eta_0'}^2\log(2)^{d-1}\frac{\bigl(\log(1 + \sigma_{\bm{n}})\bigr)^{d-1}}{\eta_0'}\\
			& = \frac{{\eta_0'}^2}{\eta_0}\bigl(\log(1 + \sigma_{\bm{n}})\bigr)^{d-1}.
		\end{align*}
		From the previous inequality, we deduce that
		\begin{equation}
			\label{ineq_left}
			\log(2)^{1-d}{\|\sigma_{\bm{n}}\|}_{\Phi_d} \ge \eta_0 = \sum_{\bm{u} = \bm{0}}^{\bm{n} - \bm{1}}\mathbb{E}[\Phi_d(d_{\bm{u}})].
		\end{equation}
		Using Lemma 3.1 in \cite{MR0365692}, for any $\lambda > 0$
		\[
		\lambda \mathbb{P}\bigl(X > \sqrt{3} \lambda\bigr) \le 3 \int_{\{X > \lambda\}}M_{\bm{n}}\mathrm{d}\mathbb{P}.
		\]
		Therefore
		\begin{equation*}
			\mathbb{E}[\Phi_d(X)] = \int_0^{+\infty}\Phi_d'(u)\mathbb{P}(X>u)\mathrm{d}u \le 3\sqrt{3}\int _{\Omega}M_{\bm{n}}\int_0^{\sqrt{3} X}\frac{\Phi_d'(u)}{u}\mathrm{d}u\mathrm{d}\mathbb{P}.
		\end{equation*}
		Computing $\int_0^{\sqrt{3} X}\frac{\Phi_d'(u)}{u}\mathrm{d}u$, we find that $\int_0^{\sqrt{3} X}\frac{\Phi_d'(u)}{u}\mathrm{d}u \le  3\sqrt{3} X \bigl(\log(1+\sqrt{3} X)\bigr)^{d-1}$. Thus
		\begin{equation}
			\label{ineq_esp}
			\mathbb{E}[\Phi_d(X)] \le 3^{\frac{d+5}{2}} \int _{\Omega}M_{\bm{n}}X\bigl(\log(1+ X)\bigr)^{d-1}\mathrm{d}\mathbb{P}.
		\end{equation}
		Applying Holder's inequality for Orlicz spaces, we get
		\[
		\int _{\Omega}M_{\bm{n}} X\bigl(\log(1+ X)\bigr)^{d-1}\mathrm{d}\mathbb{P} \le 2{\|M_{\bm{n}}\|}_{\Phi_d}{\Bigl\|X\bigl(\log(1+X)\bigr)^{d-1}\Bigr\|}_{\Psi_d}.
		\]
		Using Lemma \ref{ineq_norm_spec}, we find that
		\[
		{\Bigl\|X\bigl(\log(1+ X)\bigr)^{d-1}\Bigr\|}_{\Psi_d} \le  {\|X\|}_{\Phi_d}.
		\]
		Then
		\[
		\int _{\Omega}M_{\bm{n}} X\bigl(\log(1+X)\bigr)^{d-1}\mathrm{d}\mathbb{P} \le 2{\|M_{\bm{n}}\|}_{\Phi_d}{\|X\|}_{\Phi_d}.
		\]
		Recalling \eqref{ineq_esp}, we deduce that
		\[
		\mathbb{E}[\Phi_d(X)] \le 2\cdot3^\frac{d+5}{2}{\|M_{\bm{n}}\|}_{\Phi_d}{\|X\|}_{\Phi_d}.
		\]
		Thus, recalling that ${\|X\|}_{\Phi_d} \le 1$ and applying Lemma \ref{Lem tool}, we obtain
		\begin{equation}
			\label{ineq_right}
			\varphi_d\bigl({\|X\|}_{\Phi_d}\bigr) \le 2\cdot3^\frac{d+5}{2}{\|M_{\bm{n}}\|}_{\Phi_d},
		\end{equation}
		with $\varphi_d(x) = x\bigl(\log(1+x)\bigr)^{d-1}$ for all $x \ge 0$.
		Keeping in mind the inequalities \eqref{def_eta}, \eqref{ineq_left} and \eqref{ineq_right}, we obtain
		\begin{equation}
			\label{pos_part}
			\sum_{\bm{u} = \bm{0}}^{\bm{n} - \bm{1}}\mathbb{E}[\Phi_d(d_{\bm{u}})] \le \log(2)^{1-d}\varphi_d^{-1}\Bigl(2\cdot 3^\frac{d+5}{2}{\|M_{\bm{n}}\|}_{\Phi_d}\Bigr).
		\end{equation}
		Now, suppose that ${\|X\|}_{\Phi_d} > 1$. According to \eqref{control_sum_square}, we have
		\[
		\sum_{\bm{u} = \bm{0}}^{\bm{n} - \bm{1}}\mathbb{E}\Biggl[\Phi_d\Biggl(\frac{d_{\bm{u}}}{{\|X\|}_{\Phi_d}}\Biggr)\Biggr] \le 1.
		\]
		For any $\bm{0} \le \bm{u} \le \bm{n} - \bm{1}$ and by making use of inequality \eqref{min_phi_2}, it holds that
		\begin{align*}
			\mathbb{E}\Biggl[\Phi_d\Biggl(\frac{d_{\bm{u}}}{{\|X\|}_{\Phi_d}}\Biggr)\Biggr] & = \mathbb{E}\Biggl[\frac{d_{\bm{u}}^2}{{\|X\|}_{\Phi_d}^2}\Biggl(\log\Biggl(1+\frac{d_{\bm{u}}}{{\|X\|}_{\Phi_d}}\Biggr)\Biggr)^{d-1}\Biggr]\\
			& \ge \mathbb{E}\Biggl[\frac{\log(2)^{d-1}\Phi_d(d_{\bm{u}})}{{\|X\|}_{\Phi_d}^{d+1}\Bigl(\log\bigl(1+{\|X\|}_{\Phi_d}\bigr)\Bigr)^{d-1}}\Biggr].
		\end{align*}
		We conclude that
		\begin{equation}
			\label{ineq_left_2}
			\sum_{\bm{u} = \bm{0}}^{\bm{n} - \bm{1}}\mathbb{E}[\Phi_d(d_{\bm{u}})] \le \log(2)^{1-d}{\|X\|}_{\Phi_d}^{d+1}\Bigl(\log\bigl(1+{\|X\|}_{\Phi_d}\bigr)\Bigr)^{d-1} =: \log(2)^{1-d}\phi_d\bigl({\|X\|}_{\Phi_d}\bigr).
		\end{equation}
		where $\phi_d(x) = x^{d+1}\bigl(\log(1+x)\bigr)^{d-1}$, for all $x \ge 0$. Once again, by the same argument as in the first case, we get
		\[
		\mathbb{E}[\Phi_d(X)] \le 2\cdot3^\frac{d+5}{2}{\|M_{\bm{n}}\|}_{\Phi_d}{\Bigl\|X\bigr(\log(1+X)\bigr)^{d-1}\Bigr\|}_{\Psi_d}.
		\]
		However using Lemma \ref{ineq_norm_spec}, there exists $C_{d,\epsilon} > 0$ such that
		\[
		{\Bigl\|X\bigl(\log(1+X)\bigr)^{d-1}\Bigr\|}_{\Psi_d} \le C_{d,\epsilon}{\|X\|}_{\Phi_d}^{1+\epsilon},
		\]
		Since $\log$ is an increasing function and ${\|X\|}_{\Phi_d} > 1$, we deduce that 
		\[
		1  = \mathbb{E}\Biggl[\Phi_d\Biggl(\frac{X}{{\|X\|}_{\Phi_d}}\Biggr)\Biggr]  \le \frac{\mathbb{E}[\Phi_d(X)]}{{\|X\|}_{\Phi_d}^2}
		\]
		and so
		\[
		{\|X\|}_{\Phi_d}^2 \le \mathbb{E}[\Phi_d(X)] \le 2\cdot3^\frac{d+5}{2}C_{d,\epsilon}{\|M_{\bm{n}}\|}_{\Phi_d}{\|X\|}_{\Phi_d}^{1+\epsilon}.
		\]
		Thus
		\begin{equation}
			\label{ineq_right_2}
			{\|X\|}_{\Phi_d}^{1-\epsilon} \le 2\cdot3^\frac{d+5}{2}C_{d,\epsilon}{\|M_{\bm{n}}\|}_{\Phi_d}.
		\end{equation}
		Combining \eqref{ineq_left_2} and \eqref{ineq_right_2}, we get the following inequality
		\begin{equation}
			\label{pos_part_2}
			\sum_{\bm{u} = \bm{0}}^{\bm{n} - \bm{1}}\mathbb{E}[\Phi_d(d_{\bm{u}})] \le \log(2)^{1-d}\phi_d\circ f_{\epsilon}\Bigl(2\cdot3^\frac{d+5}{2}C_{d,\epsilon}{\|M_{\bm{n}}\|}_{\Phi_d}\Bigr).
		\end{equation}
		where $f_{\epsilon}(x) = x^{\frac{1}{1-\epsilon}}$ for all $x \ge 0$. Finally, recalling \eqref{pos_part} and \eqref{pos_part_2}, there exists $C_1,C_2 > 1$ only depending on $d$ such that
		\[
		\sum_{\bm{u} = \bm{0}}^{\bm{n} - \bm{1}}\mathbb{E}[\Phi_d(d_{\bm{u}})] \le  C_1\max\Bigl\{\varphi_d^{-1}\Bigl(C_2{\|M_{\bm{n}}\|}_{\Phi_d}\Bigr), \phi_d\circ f_{\epsilon}\Bigl(C_2{\|M_{\bm{n}}\|}_{\Phi_d}\Bigr)\Bigr\}.
		\]
		\underline{Second step:} Now suppose that $M$ can take negative values. We let 
		\[
		M_{\bm{u}}^+ = \mathbb{E}\big[\max(M_{\bm{n}}, 0)|\mathcal{G}_{\bm{u}}\big] \quad \textrm{and} \quad M_{\bm{u}}^- = \mathbb{E}\big[\max(-M_{\bm{n}}, 0)|\mathcal{G}_{\bm{u}}\big],
		\]
		with $\bm{0} \le \bm{u} \le \bm{n}$ and $ \mathcal{G}_{\bm{u}} = \sigma\bigl(M_{\bm{v}}, \bm{v} \le \bm{u}\bigr)$.	Both $M_{\bm{u}}^+$ and $M_{\bm{u}}^-$ are ortho-martingales and satisfy the conditions of the first part. Let $\bm{n} \in (\mathbb{N}^*)^d$, we define
		\[
		M_{\bm{n}}^+ := \sum_{\bm{u} = \bm{0}}^{\bm{n} - \bm{1}}d^+_{\bm{u}}, \quad M_{\bm{n}}^- := \sum_{\bm{u} = \bm{0}}^{\bm{n} - \bm{1}}d^-_{\bm{u}}, \quad \sigma_{\bm{n}}^+ = \sqrt{\sum_{\bm{u} = \bm{0}}^{\bm{n} - \bm{1}}(d_{\bm{u}}^+)^2}\quad \textrm{ and } \quad \sigma_{\bm{n}}^- = \sqrt{\sum_{\bm{u} = \bm{0}}^{\bm{n} - \bm{1}}(d_{\bm{u}}^-)^2}.
		\]
		Therefore there exists $C_1,C_2 > 1$ only depending on $d$ such that
		\begin{equation*}
			\sum_{\bm{u} = \bm{0}}^{\bm{n} - \bm{1}}\mathbb{E}[\Phi_d(d_{\bm{u}}^+)] \le  C_1\max\Bigl\{\varphi_d^{-1}\Bigl(C_2{\|M_{\bm{n}}^+\|}_{\Phi_d}\Bigr), \phi_d\circ f_{\epsilon}\Bigl(C_2{\|M_{\bm{n}}^+\|}_{\Phi_d}\Bigr)\Bigr\}
		\end{equation*}
		and
		\begin{equation*}
			\sum_{\bm{u} = \bm{0}}^{\bm{n} - \bm{1}}\mathbb{E}[\Phi_d(d_{\bm{u}}^-)] \le  C_1\max\Bigl\{\varphi_d^{-1}\Bigl(C_2{\|M_{\bm{n}}^-\|}_{\Phi_d}\Bigr), \phi_d\circ f_{\epsilon}\Bigl(C_2{\|M_{\bm{n}}^-\|}_{\Phi_d}\Bigr)\Bigr\}.
		\end{equation*}
		Using the inequalities $\Phi_d(a+b) \le 2^{d+1}(\Phi_d(a) + \Phi_d(b))$ for all $a,b \ge 0$, we obtain
		\begin{align*}
			\sum_{\bm{u} = \bm{0}}^{\bm{n} - \bm{1}}\mathbb{E}[\Phi_d(|d_{\bm{u}}|)] & \le
			2^{d+1}\sum_{\bm{u} = \bm{0}}^{\bm{n} - \bm{1}}\mathbb{E}[\Phi_d(d_{\bm{u}}^+)] + 2^{d+1}\sum_{\bm{u} = \bm{0}}^{\bm{n} - \bm{1}}\mathbb{E}[\Phi_d(d_{\bm{u}}^-)]\\
			& \le  2^{d+1}\Bigl( C_1\max\Bigr\{\varphi_d^{-1}\Bigl(C_2{\|M_{\bm{n}}^+\|}_{\Phi_d}\Bigr), \phi_d\circ f_{\epsilon}\Bigl(C_2{\|M_{\bm{n}}^+\|}_{\Phi_d}\Bigr)\Bigr\}\\
			& \qquad \qquad + C_1\max\Bigr\{\varphi_d^{-1}\Bigl(C_2{\|M_{\bm{n}}^-\|}_{\Phi_d}\Bigr), \phi_d\circ f_{\epsilon}\Bigl(C_2{\|M_{\bm{n}}^-\|}_{\Phi_d}\Bigr)\Bigr\}\Bigr)\\
			& \le 2^{d+2}C_1\max\Bigl\{\varphi_d^{-1}\Bigl(C_2{\|M_{\bm{n}}\|}_{\Phi_d}\Bigr), \phi_d\circ f_{\epsilon}\Bigl(C_2{\|M_{\bm{n}}\|}_{\Phi_d}\Bigr)\Bigr\}.
		\end{align*}
		The proof of the theorem is then complete.
	\end{proof}
	
	\section*{Acknowledgements}
	The authors would like to thank Dalibor Voln\'{y} and Christophe Cuny as well as the unknown referee for their helpful remarks during the writing of this article.
	
	\bibliographystyle{plain}
	\bibliography{FCLT_Quenched_Reding_Zhang2}

\begin{thebibliography}{10}

\bibitem{MR3483741}
D.~Barrera, C.~Peligrad, and M.~Peligrad.
\newblock On the functional {C}{L}{T} for stationary {M}arkov chains started at
  a point.
\newblock {\em Stochastic Process. Appl.}, 126(7):1885--1900, 2016.
\newblock \href{http://www.ams.org/mathscinet-getitem?mr=MR3483741}{MR3483741}.

\bibitem{MR1368394}
A.~Borodin and I.~Ibragimov.
\newblock {\em Limit theorems for functionals of random walks}.
\newblock Number 195. American Mathematical Soc., 1995.
\newblock \href{http://www.ams.org/mathscinet-getitem?mr=MR1368394}{MR1368394}.

\bibitem{MR0365692}
D.~L. Burkholder.
\newblock Distribution function inequalities for martingales.
\newblock {\em The Annals of Probability}, 1(1):19--42, 1973.
\newblock \href{http://www.ams.org/mathscinet-getitem?mr=0365692}{MR0365692}.

\bibitem{MR254912}
R.~Cairoli.
\newblock Un th\'{e}or\`{e}me de convergence pour martingales a indices
  multiples.
\newblock {\em CR Acad. Sci. Paris S{\'e}r. AB}, 269:A587--A589, 1969.
\newblock \href{http://www.ams.org/mathscinet-getitem?mr=254912}{MR254912}.

\bibitem{MR3504508}
C.~Cuny, J.~Dedecker, and D.~Voln{\`y}.
\newblock A functional clt for fields of commuting transformations via
  martingale approximation.
\newblock {\em Journal of Mathematical Sciences}, 219:765--781, 2016.
\newblock \href{http://www.ams.org/mathscinet-getitem?mr=3504508}{MR3504508}.

\bibitem{MR3178473}
C.~Cuny and F.~Merlev\`{e}de.
\newblock On martingale approximations and the quenched weak invariance
  principle.
\newblock {\em Ann. Probab.}, 42(2):760--793, 2014.
\newblock \href{http://www.ams.org/mathscinet-getitem?mr=3178473}{MR3178473}.

\bibitem{MR2886384}
C.~Cuny and M.~Peligrad.
\newblock Central limit theorem started at a point for stationary processes and
  additive functionals of reversible markov chains.
\newblock {\em J. Theoret. Probab.}, 25(1):171--188, 2012.
\newblock \href{http://www.ams.org/mathscinet-getitem?mr=2886384}{MR2886384}.

\bibitem{MR3083921}
C.~Cuny and D.~Voln\'{y}.
\newblock A quenched invariance principle for stationary processes.
\newblock {\em ALEA Lat. Am. J. Probab. Math. Stat.}, 10(1):107--115, 2013.
\newblock \href{http://www.ams.org/mathscinet-getitem?mr=3083921}{MR3083921}.

\bibitem{MR1666908}
V.~de~la Pe{\~n}a and E.~Gin{\'e}.
\newblock {\em Decoupling}.
\newblock Probability and its Applications (New York). Springer-Verlag, New
  York, 1999.
\newblock From dependence to independence, Randomly stopped processes.
  $U$-statistics and processes. Martingales and beyond.
  \href{http://www.ams.org/mathscinet-getitem?mr=1666908}{MR1666908}.

\bibitem{MR3224292}
J.~Dedecker, F.~Merlev{\`e}de, and M.~Peligrad.
\newblock A quenched weak invariance principle.
\newblock {\em Annales de l'IHP Probabilit{\'e}s et statistiques},
  50(3):872--898, 2014.
\newblock \href{http://www.ams.org/mathscinet-getitem?mr=MR3224292}{MR3224292}.

\bibitem{MR2359065}
J.~Dedecker, F.~Merlev{\`e}de, and D.~Voln{\`y}.
\newblock On the weak invariance principle for non-adapted sequences under
  projective criteria.
\newblock {\em Journal of Theoretical Probability}, 20(4):971--1004, 2007.
\newblock \href{http://www.ams.org/mathscinet-getitem?mr=MR2359065}{MR2359065}.

\bibitem{MR0566768}
C~Dellacherie and PA~Meyer.
\newblock {\em Probabilit{\'e}s et potentiel: Chapitres V \`{a} VIII:
  Th{\'e}orie des martingales, Hermann. English translation: Probabilities and
  potential. B. Theory of martingales}, volume~72.
\newblock North-Holland Mathematics Studies, 1980.
\newblock Revised edition.
  \href{http://www.ams.org/mathscinet-getitem?mr=0566768}{MR0566768}.

\bibitem{MR1826405}
Y.~Derriennic and M.~Lin.
\newblock The central limit theorem for markov chains with normal transition
  operators, started at a point.
\newblock {\em Probab. Theory Related Fields}, 119:508--528, 2001.
\newblock \href{http://www.ams.org/mathscinet-getitem?mr=MR1826405}{MR1826405}.

\bibitem{MR3522451}
M.~El~Machkouri and D.~Giraudo.
\newblock Orthomartingale-coboundary decomposition for stationary random
  fields.
\newblock {\em Stoch. Dyn.}, 16(05):1650017, 28, 2016.
\newblock \href{http://www.ams.org/mathscinet-getitem?mr=MR3522451}{MR3522451}.

\bibitem{MR3869881}
D.~Giraudo.
\newblock Invariance principle via orthomartingale approximation.
\newblock {\em Stoch. Dyn.}, 18(06):1850043, 29, 2018.
\newblock \href{http://www.ams.org/mathscinet-getitem?mr=MR3869881}{MR3869881}.

\bibitem{MR2749126}
M.~I. Gordin.
\newblock Martingale-coboundary representation for a class of stationary random
  fields.
\newblock {\em Zap. Nauchn. Sem. S.-Peterburg. Otdel. Mat. Inst. Steklov.
  (POMI)}, 364:88--108, 236, 2009.
\newblock \href{http://www.ams.org/mathscinet-getitem?mr=MR2749126}{MR2749126}.

\bibitem{MR0501277}
M.~I. Gordin and B.~A. Lif{\v{s}}ic.
\newblock The central limit theorem for stationary markov processes.
\newblock {\em Dokl. Akad. Nauk SSSR}, 239(4):766--767, 1978.
\newblock \href{http://www.ams.org/mathscinet-getitem?mr=MR0501277}{MR0501277}.

\bibitem{MR0624435}
P.~Hall and C.~C. Heyde.
\newblock {\em Martingale limit theory and its application}.
\newblock Academic press, 2014.
\newblock \href{http://www.ams.org/mathscinet-getitem?mr=0624435}{MR0624435}.

\bibitem{MR0331683}
E.~J. Hannan.
\newblock Central limit theorems for time series regression.
\newblock {\em Z. Wahrscheinlichkeitstheorie und Verw. Gebiete}, 26:157--170,
  1973.
\newblock \href{http://www.ams.org/mathscinet-getitem?mr=MR0331683}{MR0331683}.

\bibitem{MR1914748}
D.~Khoshnevisan.
\newblock {\em Multiparameter processes}.
\newblock Springer Monographs in Mathematics. Springer-Verlag, New York, 2002.
\newblock An introduction to random fields.
  \href{http://www.ams.org/mathscinet-getitem?mr=MR1914748}{MR1914748}.

\bibitem{MR0834478}
C.~Kipnis and S.~R.~S. Varadhan.
\newblock Central limit theorem for additive functionals of reversible markov
  processes and applications to simple exclusions.
\newblock {\em Comm. Math. Phys.}, 104(1):1--19, 1986.
\newblock \href{http://www.ams.org/mathscinet-getitem?mr=MR0834478}{MR0834478}.

\bibitem{MR3483738}
J.~Klicnarov{\'a}, D.~Voln{\`y}, and Y.~Wang.
\newblock Limit theorems for weighted {B}ernoulli random fields under
  {H}annan's condition.
\newblock {\em Stochastic Process. Appl.}, 126(6):1819--1838, 2016.
\newblock \href{http://www.ams.org/mathscinet-getitem?mr=3483738}{MR3483738}.

\bibitem{MR0126722}
M.~A. Krasnosel'ski{\u\i} and Ya.~B. Rutitski{\u\i}.
\newblock {\em Convex functions and Orlicz spaces}, volume 4311.
\newblock US Atomic Energy Commission, 1960.
\newblock \href{http://www.ams.org/mathscinet-getitem?mr=0126722}{MR0126722}.

\bibitem{MR0797411}
U.~Krengel.
\newblock {\em Ergodic theorems}, volume~6 of {\em De Gruyter Studies in
  Mathematics}.
\newblock Walter de Gruyter \& Co., Berlin, 1985.
\newblock With a supplement by Antoine Brunel.
  \href{http://www.ams.org/mathscinet-getitem?mr=MR0797411}{MR0797411}.

\bibitem{MR2108619}
T.~M. Liggett.
\newblock {\em Interacting particle systems}, volume~2.
\newblock Springer, 1985.
\newblock \href{http://www.ams.org/mathscinet-getitem?mr=2108619}{MR2108619}.

\bibitem{MR0293706}
Georg Neuhaus.
\newblock On weak convergence of stochastic processes with multidimensional
  time parameter.
\newblock {\em Ann. Math. Statist.}, 42:1285--1295, 1971.
\newblock \href{http://www.ams.org/mathscinet-getitem?mr=0293706}{MR0293706}.

\bibitem{MR2425365}
L.~Ouchti and D.~Voln{\`y}.
\newblock A conditional clt which fails for ergodic components.
\newblock {\em Journal of Theoretical Probability}, 21:687--703, 2008.
\newblock \href{http://www.ams.org/mathscinet-getitem?mr=2425365}{MR2425365}.

\bibitem{MR4125956}
M.~Peligrad and D.~Voln{\`y}.
\newblock Quenched invariance principles for orthomartingale-like sequences.
\newblock {\em J. Theoret. Probab.}, 33(3):1238--1265, 2020.
\newblock \href{http://www.ams.org/mathscinet-getitem?mr=4125956}{MR4125956}.

\bibitem{MR3798239}
M.~Peligrad and N.~Zhang.
\newblock {Martingale approximations for random fields}.
\newblock {\em Electronic Communications in Probability}, 23(none):1 -- 9,
  2018.
\newblock \href{http://www.ams.org/mathscinet-getitem?mr=3798239}{MR3798239}.

\bibitem{MR3409830}
Magda Peligrad.
\newblock Quenched invariance principle via martingale approximation.
\newblock {\em Asymptotic Laws and Methods in Stochastics, Fields Inst. Comm},
  76:149--165, 2015.
\newblock \href{http://www.ams.org/mathscinet-getitem?mr=3409830}{MR3409830}.

\bibitem{MR1113700}
M.~M. Rao and Z.~D. Ren.
\newblock {\em Theory of {O}rlicz spaces}, volume 146 of {\em Monographs and
  Textbooks in Pure and Applied Mathematics}.
\newblock Marcel Dekker, Inc., New York, 1991.
\newblock \href{http://www.ams.org/mathscinet-getitem?mr=MR1113700}{MR1113700}.

\bibitem{MR1182416}
D.~Stroock and B.~Zegarli\'{n}ski.
\newblock The logarithmic sobolev inequality for discrete spin systems on a
  lattice.
\newblock {\em Comm. Math. Phys.}, 149:175--193, 1992.
\newblock \href{http://www.ams.org/mathscinet-getitem?mr=MR1182416}{MR1182416}.

\bibitem{MR3427925}
D.~Voln{\`y}.
\newblock A central limit theorem for fields of martingale differences.
\newblock {\em C. R. Math. Acad. Sci. Paris}, 353(12):1159--1163, 2015.
\newblock \href{http://www.ams.org/mathscinet-getitem?mr=MR3427925}{MR3427925}.

\bibitem{MR3264437}
D.~Voln{\`y} and Y.~Wang.
\newblock An invariance principle for stationary random fields under hannan's
  condition.
\newblock {\em Stochastic Process. Appl.}, 124(12):4012--4029, 2014.
\newblock \href{http://www.ams.org/mathscinet-getitem?mr=MR3264437}{MR3264437}.

\bibitem{MR2731055}
D.~Voln{\`y} and M.~Woodroofe.
\newblock An example of non-quenched convergence in the conditional central
  limit theorem for partial sums of a linear process.
\newblock In {\em Dependence in probability, analysis and number theory}, pages
  317--322. Kendrick Press, Heber City, UT, 2010.
\newblock \href{http://www.ams.org/mathscinet-getitem?mr=MR2731055}{MR2731055}.

\bibitem{MR3157895}
D.~Voln{\`y} and M.~Woodroofe.
\newblock Quenched central limit theorems for sums of stationary processes.
\newblock {\em Statist. Probab. Lett.}, 85:161--167, 2014.
\newblock \href{http://www.ams.org/mathscinet-getitem?mr=MR3157895}{MR3157895}.

\bibitem{MR3735411}
Dalibor Volny.
\newblock Martingale-coboundary decomposition for stationary random fields.
\newblock {\em Stoch. Dyn.}, 18(2):1850011, 18, 2018.
\newblock \href{http://www.ams.org/mathscinet-getitem?mr=MR3735411}{MR3735411}.

\bibitem{Wu2005}
W.~B. Wu.
\newblock Nonlinear system theory: Another look at dependence.
\newblock {\em Proceedings of the National Academy of Sciences},
  102(40):14150--14154, 2005.

\bibitem{MR4166203}
N.~Zhang, L.~Reding, and M.~Peligrad.
\newblock On the quenched central limit theorem for stationary random fields
  under projective criteria.
\newblock {\em J. Theoret. Probab.}, 33(4):2351--2379, 2020.
\newblock \href{http://www.ams.org/mathscinet-getitem?mr=MR4166203}{MR4166203}.

\end{thebibliography}

\end{document}